\documentclass[11pt]{amsart}

\usepackage{epsfig}

\usepackage{graphicx}
\usepackage{amscd}
\usepackage{amsmath}
\usepackage{amsfonts}
\usepackage{psfrag}
\usepackage{amssymb}
\usepackage{verbatim}
\usepackage{comment}

\newenvironment{dedication}
        {\vspace{6ex}\begin{quotation}\begin{center}\begin{em}}
        {\par\end{em}\end{center}\end{quotation}}

\newtheorem{theorem}{Theorem}[section]
\theoremstyle{plain}

\newtheorem{corollary}[theorem]{Corollary}

\newtheorem{defi}[theorem]{Definition}

\newtheorem{lemma}[theorem]{Lemma}

\newtheorem{prop}[theorem]{Proposition}
\newtheorem{remark}[theorem]{Remark}

\numberwithin{equation}{section}

\newcommand \la {\lambda}

\newcommand \La {\Lambda}

\newcommand \A {{\cal A}}

\def\Yk{{\mathcal Y}}
\newcommand \q {{\bf q}}

\def\col{{\rm col}}

\newcommand \Lyap{{\rm Lyap}}

\def\Lip{{\rm Lip}}

\def\sif{{\mathfrak s}}

\def\Xx{{\mathfrak X}}
\def\Xxs{\Xx^{\vec{s}}}
\def\ek{{\mathfrak e}}
\def\Tf{{\mathfrak T}}
\def\Ff{{\mathfrak F}}
\def\ak{{\mathfrak a}}

\def\zf{{\mathfrak z}}
\def\uf{{\mathfrak u}}
\def\hf{{\mathfrak h}}
\def\One{{1\!\!1}}

\def\Hk{{\mathcal H}}
\def\Xk{{\mathcal X}}

\def\wtil{\widetilde}
\def\wt{\wtil}

\def\half{\frac{1}{2}}

\newcommand{\lam}{\lambda}
\def\Lam{\Lambda}
\newcommand{\gam}{\gamma}
\newcommand{\om}{\omega}
\newcommand{\omb}{\boldsymbol \omega}

\def\Om{\Omega}

\newcommand{\Gam}{\Gamma}
\newcommand{\sig}{\sigma}

\def\bbbe{\vec{\bf e}}

\def\AA{{\mathbb A}}

\def\Thb{{\mathbf \Theta}}

\newcommand{\R}{{\mathbb R}}

\newcommand{\Z}{{\mathbb Z}}

\def\N{{\mathbb N}}

\newcommand{\E}{{\mathbb E}\,}
\newcommand{\Prob}{{\mathbb P}\,}
\def\P{\Prob}
\newcommand{\Nat}{{\mathbb N}}

\def\bp{{\bf p}}
\def\bq{{\bf q}}

\def\Af{{A}}
\def\bqcyl{[\bq.\bq]}
\def\A{{\mathcal A}}
\def\Bk{{\mathcal B}}
\def\Rk{{\mathcal R}}

\def\Sf{{\sf S}}

\def\Vk{{\mathcal V}}

\def\one{\vec{1}}
\def\be{\begin{equation}}
\def\ee{\end{equation}}
\newcommand{\Ek}{{\mathcal E}}

\newcommand{\Wk}{{\mathcal W}}
\newcommand{\eps}{{\varepsilon}}

\newcommand{\es}{\emptyset}

\def\ov{\overline}

\newcommand{\HH}{{\mathcal H}}

\def\Max{{\rm Max}}
\def\Min{{\rm Min}}

\def\card{{\rm card}}

\def\a{a}

\def\ve1{\vec{1}}

\def\Frg{{\mathfrak G}}
\def\ok{{\mathfrak o}}
\def\M{{\mathbf M}}

\def\Tk{{\mathcal T}}
\def\Ak{{\mathcal A}}
\def\Pbi{{\mathbf \Pi}}

\def\bom{{\bf a}}
\def\loden{\underline{d}}

\begin{document}

\title [The H{\"o}lder Property for the Spectrum of Translation Flows]{The H{\"o}lder Property for the Spectrum of Translation Flows in Genus Two }

\author{Alexander I. Bufetov}
\address{Alexander I. Bufetov\\ 
Aix-Marseille Universit{\'e}, CNRS, Centrale Marseille, I2M, UMR 7373}
\address{39 rue F. Joliot Curie Marseille France }
\address{Steklov  Mathematical Institute of RAS, Moscow}
\address{Institute for Information Transmission Problems, Moscow}
\address{National Research University Higher School of Economics, Moscow}

\email{bufetov@mi.ras.ru}

\author{Boris Solomyak }
\address{Boris Solomyak\\ Department of Mathematics, Bar-Ilan University, Ramat-Gan, Israel}
\email{bsolom3@gmail.com}

\begin{abstract} The paper is devoted to generic translation flows corresponding to Abelian differentials with one zero of order two on flat surfaces of  genus two. These flows are weakly mixing by the Avila-Forni theorem.
Our main result
gives first quantitative estimates on their spectrum, establishing the H\"older property for the spectral measures of Lipschitz functions.
 The proof proceeds via uniform estimates of twisted Birkhoff integrals
in the symbolic
framework of random Markov compacta and arguments of Diophantine nature in the spirit of Salem, Erd\H{o}s and Kahane.\end{abstract}

 \maketitle

\thispagestyle{empty}


\begin{dedication}To the  memory of William Austin Veech (1938--2016)
\end{dedication}


\section{Introduction}
\subsection{Spectrum of translation flows}
Let $M$ be a compact orientable surface.  To
  a holomorphic one-form $\omb$ on $M$ one can assign  the corresponding {\it vertical} flow $h_t^+$ on $M$, i.e., the flow at unit speed along the
 leaves of the foliation $\Re(\omb)=0$.
The vertical flow preserves the measure 
$
{\mathfrak m}=i(\omb\wedge {\overline \omb})/2, 
$
the area form induced by $\omb$. 
By a theorem of Katok \cite{katok}, the flow $h_t^+$ is never mixing.
The moduli space of abelian differentials carries a natural volume measure, called the Masur-Veech measure \cite{masur}, \cite{veech}.  
For almost every Abelian differential with respect to the Masur-Veech measure,  Masur \cite{masur} and Veech \cite{veech}  independently and simultaneously proved that the flow   $h_t^+$ is
uniquely ergodic.  Under additional assumptions on the combinatorics of the abelian differentials, weak mixing for almost all  translation flows has been established by Veech in \cite{veechamj} and in full generality by Avila and Forni
\cite{AF}. The spectrum of translation flows is therefore almost surely  continuous and always has a singular component. No quantitative results have, however,  previously been obtained about the spectral measure. 
Sinai [personal communication] posed the following 

\medskip

{\bf Problem.} {\it Find the local asymptotics for the spectral measure of translation flows. }

\medskip
\subsection{Formulation of the main result.}
The aim of this paper is to give first quantitative estimates on the spectrum of translation flows. 
Let $\HH(2)$ be the moduli space of abelian differentials on a surface of genus $2$ with one zero of order two. The  natural smooth Masur-Veech measure on the stratum $\HH(2)$ is denoted by $\mu_2$.
Our main result is
that for almost all abelian differentials in $\HH(2)$, the spectral measures of Lipschitz functions with respect to the corresponding translation flows have the H{\"o}lder property. 
Recall that for a square-integrable test function $f$, the spectral measure $\sig_f$ is defined by
$$
\widehat{\sig}_f(-t) = \langle f\circ h_t^+, f \rangle,\ \ \ t\in \R,
$$
see Section 3.3 for details.
A point mass for the spectral measure corresponds to an eigenvalue, so H\"older estimates for spectral measures quantify weak mixing for our systems.


\begin{theorem}\label{main-moduli}
There exists $\gamma>0$ such that  for $\mu_2$-almost every abelian differential $(M, \omb)\in \HH(2)$ the following holds. For any $B>1$ there exist constants $C=C(\omb,B)$ and $r_0=r_0(\omb,B)$ such that for any Lipschitz function $f$ on $M$, for all $x\in [B^{-1},B]$ and $r\in (0, r_0)$ we have 
$$
\sig_f([x-r, x+r])\le C\|f\|_L\cdot r^\gamma.
$$
\end{theorem} 

The proof uses the  symbolic formalism of \cite{Buf-umn}, namely, the representation of translation flows by flows along orbits of the asymptotic equivalence relation of a Markov compactum. The H{\"o}lder property 
is then reformulated as a consequence of a statement on Diophantine approximation involving the incidence matrices of the graphs forming the Markov compactum that codes the translation flow.
These incidence matrices are, in turn, realizations of a renormalization cocycle, isomorphic, under our symbolic coding, to the Kontsevich-Zorich cocycle over the Teichm{\"u}ller flow on the moduli space of abelian differentials.  By the Oseledets multiplicative ergodic theorem, the cocycle admits 
a decomposition into Oseledets subspaces corresponding to the distinct Lyapunov exponents. 
Our argument in this paper requires that the Kontsevich-Zorich cocycle admit  exactly two 
positive Lyapunov exponents and not have zero Lyapunov exponents; this is the reason why  our main result only covers the stratum $\HH(2)$. 

We stress that our proof only works for the Masur-Veech smooth measure. This can informally be explained as follows. A translation flow can be represented as a suspension flow over an interval exchange transformation with a piecewise constant roof function. Take a measure on a stratum, invariant under the action of the Teichm{\"u}ller flow. If one fixes an interval exchange transformation, then the invariant measure
yields a conditional measure on the polyhedron of admissible tuples of heights of the rectangles.
Our argument in its present form requires that this measure be absolutely continuous with respect to the natural Lebesgue measure; this property only holds for the Masur-Veech smooth measure. 

The H\"older exponent $\gam>0$ obtained in Theorem 1.1 can be given explicitly, but it is very small and certainly not sharp, so we do not pursue this. It is an interesting question to determine the sharp exponent. More can be said about the H\"older exponent at $x=0$ (assuming the test function $f$ has zero mean), see Remark~\ref{rem-Forni}.
Local estimates for spectral measures are obtained via growth estimates of twisted Birkhoff integrals. 

The paper is organized as follows. In Section 2, we introduce the necessary setup of Markov compacta and Bratteli-Vershik (BV) transformations, as well the alternative, closely related framework, based on sequences of substitutions, or $S$-adic systems (see e.g.\ \cite{Vershik1,Vershik2,Vershik-Livshits,solomyak,Buf-umn,BST} for futher background). We will be working in this ``symbolic'' framework for most of the paper, only returning to translation flows in the last Section 11. Estimates of twisted Birkhoff integrals in terms of matrix product are considered in Section 3, which builds on our paper \cite{BuSol2}. The difference is that  \cite{BuSol2} was concerned with a single substitution, or equivalently, with a stationary Bratteli-Vershik diagram.
In Section 4 we state the main theorem for {\em random} BV-systems, satisfying a certain list of conditions, 
among which the key condition is a uniform large deviations  estimate.
That theorem is proved in Sections 5-10. 
It should be emphasized that there are substantial technical difficulties in the transition from the stationary framework of \cite{BuSol2} to the non-stationary setting of this paper. 
In the case of a 
single substitution matrix we could rely on estimates of the Vandermonde matrix, its determinant and its inverse. 
In this paper, we need similar estimates for the cocycle matrices. The Oseledets Theorem  controls norms of these matrices and angles between different subspaces only up to subexponential errors.  It is precisely in order to control these errors that we need the assumption that only two Lyapunov exponents be positive.
In Section 10 we use a variant of the ``Erd\H{o}s-Kahane argument'' that had
originated in the theory of Bernoulli convolution measures, see \cite{Erd,Kahane}.
Finally, in Section 11 we conclude the proof by explaining how the symbolic coding of translation flows gives rise to suspension flows over BV-maps and by checking all the conditions required. The key probabilistic condition is derived from a (slight generalization of) large deviation estimate
for the Teichm\"uller flow in \cite{buf-jams}.

\subsection{General remarks} For comparison, consider the work on weak mixing by Avila-Forni \cite{AF} and the recent  Avila-Delecroix \cite{AD}. Let us parametrize all possible eigenvalues by the line in $H^1(M;\R)$ through the imaginary part of the 1-form $\omb$ (we borrow here some of the wording from the introduction of \cite{AD}, and also refer to that paper for the explanation of the terms used in this paragraph). The Veech criterion \cite{veechamj} says that an eigenvalue is parametrized by an element of ``weak-stable lamination'' associated to an acceleration of the Kontsevich-Zorich cocycle acting modulo $H^1(M,\Z)$. 
In \cite{AF}, a probabilistic method is applied to exclude non-trivial intersections of an arbitrary fixed line in $H^1(M;\R)$ with the weak stable foliation to prove weak mixing for almost every interval exchange, and a simpler, ``linear exclusion'' is used to exclude such intersections and prove weak mixing for almost every translation flow. (On the other hand, \cite{AD} uses additional structure to prove weak-mixing of the translation flow in almost every direction for a given
non-arithmetic Veech surface.) In order to prove H\"older continuity, roughly speaking, we need that there is a positive density of iterates that fall {\bf outside} of a fixed neighborhood of $H^1(M,\Z)$ along the orbit of the acceleration of the Kontsevich-Zorich cocycle, uniformly for a fixed line in $H^1(M;\R)$. This requires much more delicate estimates, see Sections 9 and 10, hence more limited scope of our results. 

Our inspiration came from the work of Salem, Erd\H{o}s, and Kahane on Diophantine approximation and Bernoulli convolutions.  A well-known problem, open since the 1930's, asks whether 
\be \label{Pisot}
\{\lam>1:\ \exists\, \alpha>0\  \  \mbox{such that }\  \ \lim_{n\to \infty} \|\alpha\lam^n\|_{\R/\Z} =0 \}
\ee
is the set of PV-numbers,  that is, $\lam$ an algebraic integer all of whose conjugates are less than one in absolute value. Salem \cite{Salem} showed that the set in (\ref{Pisot}) is countable.
In short, if \be \label{Salem}
\alpha \lam^n = K_n + \eps_n,
\ee
with $K_n$ being the nearest integer, is such that 
$|\eps_{n+j}|$ for $j\ge 0$, is sufficiently small, then all $K_{n+j},\ j\ge 2,$ are uniquely determined by $K_n, K_{n+1}$, and $\lam$ may be recovered as $\lim_{j\to \infty} \frac{K_{n+j}}{K_{n+j-1}}$, hence there are only countable many possibilities for $\lam$. This may be compared with the ``linear exclusion'' from \cite{AF} (in the sense that in both cases the question of (non)convergence is addressed; the method is actually different).

The Bernoulli convolution measure $\nu_\lam$ can be defined by the formula for its Fourier transform:
$$
\widehat{\nu}_\lam(t) = \prod_{n=0}^\infty \cos(2\pi \lam^{-n} t),\ \ t\in \R.
$$
Erd\H{o}s \cite{Erd0} observed that $\lim_{t\to \infty} |\widehat{\nu}_\lam(t)| \ne 0$ for PV-numbers $\lam$, which implies that the corresponding Bernoulli convolution measure $\nu_\lam$ is singular. Salem (see \cite{Salem-book}) proved that $\lim_{t\to \infty} |\widehat{\nu}_\lam(t)| = 0$ in all other cases, but it is an open problem whether such $\nu_\lam$ is necessarily absolutely continuous (see \cite{Hochman,pablo,Varju} for the recent progress in this direction). However, in another early paper, Erd\H{o}s \cite{Erd} proved that for almost all $\lam$ , $\widehat{\nu}_\lam(t)=O(|t|^{-C(\lam)})$, as $|t|\to \infty$, with $C(\lam)\to\infty$ as
$\lam\to 1$, which implies absolute continuity with arbitrarily high smoothness for $\lam$ sufficiently close to one (Kahane \cite{Kahane} later showed that in the argument of Erd\H{o}s one can replace ``almost all'' by ``all outside of  a set of any fixed positive Hausdorff dimension). The argument of Erd\H{o}s and Kahane builds on the Salem idea: if all $\eps_n$ in (\ref{Salem}), outside a set of $n$'s of small positive density are sufficiently small in absolute value, then ``most'' of the $K_n$'s are uniquely determined by the previous ones, and we get a good covering of the ``bad set'' to yield the dimension estimate. A variation on this idea, which came to be called the ``Erd\H{o}s-Kahane argument'',  was used in \cite{BuSol2}, and a much more delicate version of it is developed in this work.

\medskip

{\bf Acknowledgements.} We are  deeply grateful to Giovanni Forni for our useful discussions of the preliminary version of the paper, and to the referee for constructive criticism and helpful suggestions.
The research of A. Bufetov on this project has received funding from the European Research Council (ERC) under the European Union's Horizon 2020 research and innovation programme under grant agreement No 647133 (ICHAOS).
It was also supported by the A*MIDEX project (no. ANR-11-IDEX-0001-02) funded by the  programme ``Investissements d'Avenir'' of the Government of the French Republic, managed by the French National Research Agency (ANR), by the Grant MD 5991.2016.1 of the President of the Russian Federation and by  the Russian Academic Excellence Project `5-100'.
B. Solomyak  has been supported  by NSF grants DMS-0968879, DMS-1361424, and the ISF grant 396/15.

\section{Preliminaries}

\subsection{Markov compacta and Bratteli-Vershik transformations}
The reader is referred to \cite{Vershik1,Vershik2,Vershik-Livshits,solomyak,Buf-umn} for further background.

Let $\Frg$ be the set of all oriented graphs on $m$ vertices such that there is an edge starting at every vertex and an edge ending at every vertex (we allow loops and multiple edges). For an edge $\ek$ we denote by $I(\ek)$ and $F(\ek)$ the initial and final vertices of $\ek$ correspondingly. Assume we are given a sequence $\{\Gam_n\}_{n\in \Z}$ of graphs belonging to $\Frg$. To this sequence we associate the {\em Markov compactum} of paths in our sequence of graphs:
\be \label{def-Mark}
X = \{\ov{\ek} = (\ek_n)_{n\in \Z}:\ \ek_n \in \Ek(\Gam_n),\ F(\ek_{n+1}) = I(\ek_{n})\}.
\ee
We will also need the one-sided Markov compactum $X_+$ (respectively $X_-$), defined in the same way with elements $(\ek_n)_{n\ge 1}$ (respectively $(\ek_n)_{n\le 0}$). A one-sided sequence of
graphs in $\Frg$ can also be considered as a {\em Bratteli diagram} of (finite) rank $m$. We then view its vertices as being arranged in levels $n\ge 0$, so that
the graph $\Gam_n$ connects the vertices of level $n$ to vertices of level $n-1$. (Note that in some papers the direction of the edges is reversed.)
Let $A_n(X) = A(\Gam_n)$ be the incidence matrix of
the graph $\Gam_n$ given by the formula 
$$
A_{ij}(\Gam) = \#\{\ek\in \Ek(\Gam):\ I(\ek) = i,\ F(\ek) = j\}.
$$

On the Markov compactum $X$ we define the ``future tail'' and ``past tail'' equivalence relations, in which
two infinite paths are equivalent iff they agree from some point on (into the future or into the past).

There is a standard construction of {\em telescoping} (= aggregation): for any sequence $1=n_0<n_1 < n_2< \cdots$ we ``concatenate'' the graphs
$\Gam_{n_j},\ldots, \Gam_{n_{j+1}-1}$ to obtain $\Gam'_j\in \Frg$. 

\medskip

\noindent
{\bf Standing Assumption.}  {\em The sequence $\Gam_n$ (after appropriate
telescoping) contains infinitely many occurrences of a single graph $\Gam$ with a strictly positive incidence matrix, both in the past and in the future. }

\medskip

In this case, as is
well-known since the work   of Furstenberg (see e.g.\ (16.13) in \cite{furst}), the 
Markov compactum $X_+$ is {\em uniquely ergodic}, which means that there is a unique invariant probability measure for the ``future tail'' equivalence relation.
We need to develop this point in more detail. In fact, in this case we have (see \cite[1.9.5]{Buf-umn}) that there exist strictly positive vectors $\vec{\zf}^{(\ell)}, \vec{\uf}^{(\ell)}$, for $\ell\in \Z$, such that
$$\vec{\zf}^{(\ell)} = A_\ell^t \,\vec{\zf}^{(\ell+1)},\ \ \ell\in \Z;$$
\be \label{equiv2}
\bigcap_{n\in \N} A^t_{\ell+1} \cdots A^t_{\ell+n} \R_+^m =\R_+\,\vec{\zf}^{(\ell)}, \ \ \ell\in \Z;
\ee
$$\vec{\uf}^{(\ell)} = A_\ell \,\vec{\uf}^{(\ell-1)},\ \ \ell\in \Z;$$
$$\bigcap_{n\in \N} A_{\ell-1} \cdots A_{\ell-n} \R_+^m =\R_+\,\vec{\uf}^{(\ell)}, \ \ \ell\in \Z.$$
The vectors are normalized by $$|\vec{\zf}^{(0)}|_1=1,\ \langle \vec{\zf}^{(0)}, \vec{\uf}^{(0)} \rangle =1.$$
As already mentioned, the Markov compactum $X_+$ is then uniquely ergodic, with the unique tail-invariant probability measure $\nu_+$ given by 
\be \label{meas1}
\nu_+(X_j^+) = \zf^{(0)}_j,\ \ \ \nu_+([\ek_1\ldots\ek_\ell]) = \zf^{(\ell)}_{I(\ek_\ell)},
\ee
where $X_j^+$ is the set of one-sided paths $ \ek_1\ek_2\ldots \in X_+$ such that $F(\ek_1)=j$ and $[\ek_1\ldots\ek_\ell]$ is the cylinder set corresponding to the finite path. 
The advantage of working with 2-sided Bratteli diagrams, which is one of the key ideas of \cite{Buf-umn}, is that one can similarly define the
``dual'' measure $\nu_-$ on the set of ``negative paths'' $X_-$, invariant under the ``past tail'' equivalence relation. Then $\nu = \nu_+\times \nu_-$ is a probability measure on $X$.

Now suppose that a linear ordering (Vershik's ordering) is defined on the set $\{\ek\in \Ek(\Gam_\ell):\ I(\ek) = i\}$ for all $i\le m$ and
$\ell\in \Z$. This induces a partial lexicographic ordering $\ok$ on the Markov compactum $X$; more precisely, two paths are comparable if they agree for $n\ge t$ for some $t\in \Z$. 
(Also two paths in $X_-$ are comparable if they end at the same vertex.) Restricting this ordering to the 1-sided compactum $X_+$, we obtain the {\em adic}, or
{\em Bratteli-Vershik} (BV) transformation ${\mathfrak T}$, defined as the immediate successor of a path $\ov{\ek}$ in the ordering $\ok$. (See also the work of S.\ Ito \cite{ito}, where a similar construction had appeared earlier.) Let $\Max(\ok)$ be the set of paths in $X_+$ such that its every edge is maximal. It is easy to see that $\card(\Max(\ok))\le m$. Similarly define the set of minimal paths $\Min(\ok)$,
and let $\wtil{X}_+$ be the set of paths, which are not tail equivalent to any of the paths in $\Min(\ok)\cup
\Max(\ok)$. Then the $\Z$-action $\{\Tf^n\}_{n\in \Z}$ is well-defined on $\wtil{X}_+$. Since we are 
excluding a countable set, the action is defined almost everywhere with respect to any non-atomic measure;
certainly, $\nu_+$-a.e.\ in the uniquely ergodic case. We similarly define  $\wtil{X}$ as the set of bi-infinite paths in $X$ which are not forward tail-equivalent to any of the maximal or minimal paths.
Note that invariant measures for the future tail equivalence relation are precisely the invariant measures for the BV map, so we get unique ergodicity of the system $(X_+,\Tf)$ under our standing assumption.
An easy, but important, fact is that the construction of telescoping/aggregation preserves the Vershik ordering, and the
corresponding BV-systems are isomorphic. 

We shall also consider suspension flows over BV-systems, with a piecewise-constant roof function
depending only on the vertex at the level 0. More precisely, let $X_+$ be a one-sided Markov compactum with a Vershik ordering and BV-transformation $\Tf$.
For a strictly positive vector $\vec{s} = (s_1,\ldots,s_m)$ define the roof function $\phi_{\vec{s}}$ to be  equal to $s_j$ on the cylinder set $X_j^+$. The resulting space will be denoted $\Xx^{\vec{s}}$:
$$
\Xx^{\vec{s}}:= \bigsqcup_{j=1}^m X_j^+\times [0,s_j] \ \mbox{\bf \large /}\!\!\sim\,,\ \ \ \mbox{with}\ \ 
(\ek,\phi_{\vec{s}}(\ek))\sim (\Tf\ek,0),
$$
on which we consider the usual suspension flow $\{h_t\}_{t\in \R}$.
It preserves the measure induced by the $\Tf$-invariant measure $\nu_+$ on $X_+$ and the Lebesgue measure on $\R$.
We will need the following

\begin{lemma}\cite[Section 2.5]{Buf-umn} \label{symb-flow}
Given a Vershik ordering on a uniquely ergodic Markov compactum $X$ with the unique invariant measure
$\nu$, there is a symbolic flow 
$(X,h_t^+)_{t\in \R},$ defined on $\wtil{X}$, so 
  $\nu$-a.e.\ on $X$, which is measurably conjugate to the suspension flow $(\Xx^{\vec{s}},h_t)_{t\in \R}$ over $(X_+, \Tf)$, with $s_j=\uf_j^{(1)},\ j\le m$. Moreover , the conjugating map $\Ff:\,X\to \Xx^{\vec{s}}$ is given by
$$
\Ff(\ov{\ek}) = (P_+\ov{\ek},t),\ \ \ \mbox{where}\ \ 
t = \nu_-(\{\ak\in X_-:\ I(\ak_0) = F(\ek_1),\ \ak \preccurlyeq P_-\ek\}),
$$
where $P_+, P_-$ are the truncations from $X$ to $X_+$, $X_-$ respectively and $\preccurlyeq$ is the 
Vershik order. The map $\Ff$ is well-defined on $\wtil{X}$ and its inverse is well-defined over $\wtil{X}_+$.
\end{lemma}

\subsection{Weakly Lipschitz functions.}
Following \cite{Bufetov1, Buf-umn}, we consider the space of {\em weakly Lipschitz functions}
on a uniquely ergodic Markov compactum $X$ with the  probability measure $\nu_+$ invariant for the ``forward tail" equivalence relation and a Vershik ordering $\ok$. Recall that $\wtil{X}$ denotes the set of paths in $x$ that are not
(forward) tail equivalent to any of the maximal or minimal paths in the Vershik ordering.
We say that $f$ is weakly Lipschitz and write
$f\in \Lip_w^+(X)$ if $f$ is defined and bounded on $\wtil{X}$, and there exists $C>0$ such that for all $\ov{\ek},\ov{\ek}'\in \wtil{X}$, satisfying
$\ek_k = \ek'_k$ for $-\infty<k\le n$, with $n\in \N$, we have
\be \label{LL1}
|f(\ov{\ek})-f(\ov{\ek}')|\le C \nu_+([\ek_1\ldots \ek_n]).
\ee
The norm in $\Lip_w^+(X)$ is defined by
\be \label{Lip-norm}
\|f\|_L:= \|f\|_{\sup} + \wtil{C},
\ee
where $\wtil{C}$ is the infimum of the constants in (\ref{LL1}). Note that a weakly Lipschitz function is mapped into a weakly Lipschitz function when telescoping of the diagram is performed, and this does not increase the norm $\|\cdot\|_L$.

We analogously define the space $\Lip_w^+(\Xx^{\vec{s}})$ of weakly Lipschitz functions on the 
space $\Xx^{\vec{s}}$ of the suspension flow over $(X_+,\Tf)$ with the roof function determined by the 
vector $\vec{s}\in \R^m_+$. Namely, $f\in \Lip_w^+(\Xx^{\vec{s}})$ if it is defined and bounded on $\wtil{X}_+$ and
there exists $C>0$ such that 
for all $\ov{\ek},\ov{\ek}'\in \wtil{X}_+$, 
satisfying
$\ek_k = \ek'_k$ for $k\le n$, with $n\in \N$,  and all $t\in [0, s_{F(\ek_1)}]$ we have
\be \label{LL2}
|f(\ov{\ek},t)-f(\ov{\ek}',t)|\le C \nu_+([\ek_1\ldots \ek_n]).
\ee
The norm $\|f\|_L$ is defined as in (\ref{Lip-norm}). 

Note that weakly Lipschitz functions are not 
Lipschitz in the ``transverse'' direction, corresponding to the ``past'' in the 2-sided Markov compactum and to the vertical direction in the space of the suspension flow. Note also that if $f\in \Lip_w^+(X)$, then
$f\circ \Ff^{-1}\in \Lip_w^+(\Xx^{\vec{s}})$, with the same norm, where $\Ff$ is defined in Lemma~\ref{symb-flow}.

\subsection{Substitutions.}
Along with the Markov compactum and BV-transformation, it is convenient to use the language of {substitutions} (see e.g.\ \cite{Fogg} for further background). Consider the
alphabet $\Ak = \{1,\ldots,m\}$, which is identified with the vertex set of all the graphs $\Gam_n$. A {\em substitution}  is a map $\zeta:\,\A\to \A^+$, extended to  $\A^+$ and $\A^{\N}$ by concatenation. Given a Vershik ordering $\ok$ on a 1-sided Bratteli diagram
$\{\Gam_{j}\}_{j\ge 1}$,  the substitution $\zeta_j$ takes every $b\in \A$ into the word in $\A$ corresponding to all the vertices to which
there is a $\Gam_{j}$-edge starting at $b$, in the order determined by $\ok$. Formally, the {\em length} of the word $\zeta_j(b)$ is
$$
|\zeta_j(b)| = \sum_{a=1}^m A_{b,a}(\Gam_j),
$$
and the substitution itself is given by 
\be \label{def-subs1}
\zeta_j(b) =u_1^{b,j}\ldots u_{|\zeta_j(b)|}^{b,j},\ \ \ b\in \A,\ \ j\ge 0,
\ee
where $(b, u_i^{b,j})\in \Ek(\Gam_j)$, listed in the linear order prescribed by $\ok$.
Substitutions, extended to $\Ak^+$, can be composed in the usual way as transformations $\Ak^+\to \Ak^+$. We will use the notation
$$
\zeta^{[n_1,n_2]} := \zeta_{n_1} \circ \cdots \circ \zeta_{n_2},\ \ n_2 \ge n_1,
$$
and
$$
\zeta^{[n]} := \zeta^{[1,n]},\ \ n\ge 1.
$$
Given a substitution $\zeta$, its {\em substitution matrix} is defined by
$$
\Sf_\zeta(a,b):= \# \mbox{\ symbols $a$ in $\zeta(b)$}.
$$
Observe that $\Sf_{\zeta_1\circ \zeta_2} = \Sf_{\zeta_1}\Sf_{\zeta_2}$. We will denote $\Sf_n = \Sf_{\zeta_n}$. Notice that the transpose 
$\Sf_n^t$ is exactly the incidence matrix $A_n = A(\Gam_n)$ by the definition of  $\zeta_n$:
\be \label{incid}
\Sf_{\zeta_n}^t = A(\Gam_n). 
\ee
 We will use the notation
$$
\Sf^{[n_1,n_2]}:= \Sf_{\zeta^{[n_1,n_2]}}\ \ \ \mbox{and}\ \ \ \Sf^{[n]}:= \Sf_{\zeta^{[n]}}.
$$

Next, we associate to any $\ov{\ek}\in X_+$, its ``horizontal sequence'' in the alphabet $\Ak$, defined by
$$
\hf(\ov{\ek}):=x = (x_n)_n,\ \ \mbox{where}\ \ x_n=x_n(\ov{\ek})= b\ \ \mbox{whenever}\ F(\Tf^n(\ov{\ek})_1)=b,\ \ n\in \Z,
$$
that is, we keep track of the vertex at level zero, applying the BV-transformation.
The horizontal sequence is just the symbolic dynamics of $\Tf$ with respect to the 0-level cylinder partition. We get a full 2-sided infinite sequence $\hf(\ov{\ek})=x=(x_n)_{n\in \Z}$ whenever the (2-sided) orbit of $\ov{\ek}$ under $\Tf$ does not hit a maximal
or a minimal path. (By our assumptions, no orbit of $\Tf$ can hit both.) Thus $\hf: \wtil{X}_+\to \Ak^\Z$ is well-defined. We obtain, by definition, the following commutative diagram:
$$
\begin{CD}
\wtil{X}_+ @>{\Tf}>> \wtil{X}_+ \\
@V{\hf}VV @VV{\hf}V \\
\Ak^\Z @>T>> \Ak^\Z
\end{CD}
$$
where $T$ is the left shift on $\Ak^\Z$. 
Of course, the map $\hf$ is far from being surjective. In order to understand its image, it is useful
to have a more explicit algorithm for $\hf(\ek)$. Suppose that the path $\ek\in X_+$ goes through the vertices $b_0, b_1,\ldots$, that is,
$b_n = F(\ek_{n+1})$. Recalling the definition of the substitutions $\zeta_n$ we can write
\be \label{repka1}
\zeta_n(b_{n}) = u_{n-1} b_{n-1} v_{n-1},\ \ n\ge 1,
\ee
where $u_{n-1}$ and $v_{n-1}$ are words, possibly empty. Note that there may be more than one occurrence of $b_{n-1}$ in $\zeta_n(b_{n})$, but
we choose the representation (\ref{repka1}) according to the edge $\ek_n$. 
Consider the following sequence of words $U_n$, $n\ge 0$, defined inductively.
We start with 
$$
U_0 = u_0\mbox{\bf .}b_0 v_0,
$$
where the dot {\bf .}\ separates negative and positive coordinates. Then $U_{n+1}$ is obtained from $U_n$ inductively, by appending $\zeta_{n+1}(u_{n+1})$  from
the left and $\zeta_{n+1}(v_{n+1})$ from the right.
If we disregard the location of the dot, we simply have 
$$U_n = \zeta_1 \circ \cdots \circ \zeta_{n+1}(b_{n+1}) = \zeta^{[n+1]}(b_{n+1}),\ \ n\ge 0.
$$
When we take the location of the dot into account, typically, the words $U_n$ will ``grow'' to infinity, both left and right, to a limiting 2-sided sequence, which is exactly $\hf(\ek)$:
\be \label{hff}
\hf(\ek) = \ldots \zeta_2(u_2)\zeta_1(u_1)u_0{\mbox{\bf .}}b_0v_0\zeta_1(v_1)\zeta_2(v_2)\ldots 
\ee
 The other alternative is that it grows to infinity only on one side, which happens if and only if $\ek$ is
tail equivalent to either minimal or a maximal path. Denote
$$
\Yk:={\rm clos}(\hf(\wt{X}_+)),
$$
where ``clos'' denotes the closure in the product topology.
Now the following is clear.

\begin{lemma} \label{lem-image}
The space $\Yk\subset \Ak^\Z$ is exactly the set of 2-sided sequences $x$ with the property that any subword of $x$ appears as a subword of $\zeta^{[n]}(b)$ for some $b\in \Ak$ and $n\ge 1$.
\end{lemma}


\noindent {\bf Remark.} Dynamical systems $(\Yk,T)$ have been studied under the name of {\em $S$-adic systems}.
They were originally introduced by S. Ferenczi \cite{Feren}, with the additional assumption that there are finitely many different substitutions in the sequence $\{\zeta_j\}$; however, more recently this restriction has been removed, see e.g.\ \cite{BST}.

\medskip

Let $h^{(n)}_i$ be the number of finite paths $\ek_1,\ldots,\ek_n$ such that $I(\ek_n)=i$. (This is the height of the Rokhlin tower for the BV-map.) We get a sequence of real vectors
$\vec{h}^{(n)}$ which satisfy the equations:
\be \label{equiv11}
\vec{h}^{(n+1)} = A_{n+1} \vec{h}^{(n)} = \Sf_{n+1}^t \vec{h}^{(n)},\ \ \ h^{(n)}_i = |\zeta^{[n]}(i)|,\ \ n\ge 1.
\ee

\medskip

Let $\sif$ be the left shift transformation on the 1-sided Markov compactum: $\sif(\ek_1,\ek_2,\ek_3,\ldots) = (\ek_2,\ek_3,\ldots)$. We thus obtain a sequence of Markov compacta
$X_+^{(\ell)}$ for $\ell\ge 0$, with $X_+^{(0)} := X_+$, so that $\sif:\,X_+^{(\ell)}\to X_+^{(\ell+1)}$ for all $\ell$. The Vershik ordering and BV-transformation are
naturally induced on the whole family. We can then consider the horizontal sequence map
$$
\hf:\,\wt{X}_+^{(\ell)}\to \Ak^\Z.
$$
Its image, denoted by $\Yk^{(\ell)}$, is described similarly to $\Yk=\Yk^{(0)}$, as the set of all sequences in $\Ak^\Z$ whose every subword occurs as a block in 
$\zeta^{[\ell+1,n]}(b)$ for some $n\ge \ell+1$ and $b\in \Ak$.

\medskip

A substitution $\zeta$ acts on $\Ak^\Z$ as follows:
$$
\zeta(\ldots a_{-1}a_0\mbox{\bf .}a_1\ldots) = \ldots \zeta(a_{-1})\zeta(a_0)\mbox{\bf .}\zeta(a_1)\ldots
$$
Definitions imply that we have a sequence of surjective maps
$
\zeta_\ell:\,\Yk^{(\ell)} \to \Yk^{(\ell-1)},\ \ \ell\ge 1.
$
It follows from definitions (and the explicit formulas for $\hf$) that
$$
x = T^{k-1}\zeta_1(x'),
$$
where $k$ is the number (rank) of the edge $\ek_0$ in the Vershik ordering. Of course, similar formulas relate $\hf(\ek)$ and $\hf(\sif\ek)$ for $\ek\in X_+^{(\ell)}$.
(Recall that $T$ denotes the left shift on $\Ak^\Z$.)


\section{Estimating the growth of exponential sums and matrix products}

We use the following convention for the Fourier transform of functions and measures: given $\psi\in L^1(\R)$ we set
$
\widehat{\psi}(t) = \int_\R e^{-2\pi i \om t} \psi(\om)\,d\om,
$
and for a  probability measure $\nu$ on $\R$ we let
$
\widehat{\nu}(t) = \int_\R e^{-2\pi i \om t}\,d\nu(\om).
$

\subsection{Spectral measures and twisted Birkhoff integrals}
Since our goal is to obtain estimates of spectral measures, we recall how they are defined for flows.
Given a measure-preserving flow
$(Y, h_t,\mu)_{t\in \R}$ and a test function $f\in L^2(Y,\mu)$, 
there is a finite positive Borel measure $\sig_f$ on $\R$ such that
$$
\widehat{\sig}_f(-\tau)=\int_{-\infty}^\infty e^{2 \pi i\om \tau}\,d\sig_f(\om) = \langle f\circ h_\tau, f\rangle\ \ \ \mbox{for}\ \tau\in \R.
$$
In order to obtain local bounds on the spectral measure, we can use growth estimates of the ``twisted Birkhoff integral''
\be \label{twist1}
S_R^{(y)}(f,\om) := \int_0^R e^{-2\pi i \om \tau} f\circ h_\tau(y)\,d\tau.
\ee
The following lemma is standard; a proof may be found in \cite[Lemma 4.3]{BuSol2}.

\begin{lemma} \label{lem-varr} Suppose that for some fixed $\om \in \R$, $R_0>0$, and $\alpha \in (0,1)$ we have
\be \label{L2est}
\left\|S_R^{(y)}(f,\om)\right\|_{L^2(Y,\mu)}\le C_1R^\alpha\ \ \mbox{for all}\ R\ge R_0.
\ee
Then 
\be \label{locest}
\sig_f([\om-r,\om+r]) \le \pi^2 2^{-2\alpha} C_1^2 r^{2(1-\alpha)}\ \ \mbox{for all}\ r \le (2R_0)^{-1}.
\ee
\end{lemma}

\begin{remark} \label{rem-Forni} {\em 1. Since $(Y,\mu)$ is a probability space, the $L^2$-estimate (\ref{L2est}) obviously follows from a uniform estimate. 
We only need $L^2$ estimates for the proof of our main result. Nonetheless, we expect that the uniform estimates of this paper  would have further applications.

2. Estimates of twisted Birkhoff sums have been used for a number of different dynamical systems recently; in particular, see the work of Forni and Ulcigrai \cite{FU} on the Lebesgue spectrum for smooth time changes of the horocycle flow.

3. For $\om=0$ the expression (\ref{twist1}) reduces to the usual Birkhoff integral, for which sharp estimates and asymptotics are known (under the assumption that the test function $f$ has zero mean) in a number of cases. 
It should be possible to obtain precise asymptotics of the spectral measure at zero for almost every translation flow in the context of Theorem~\ref{main-moduli}, even for a general stratum in genus $g>1$. We expect that it is 
governed by the second Lyapunov exponent, based on the results of \cite{Bufetov1}, analogously to \cite[Theorem 6.2]{BuSol2}

4. It is easy to see that $\left\|S_R^{(y)}(f,\om)\right\|_{L^2(Y,\mu)} = O(R)$, and $o(R)$ indicates the absence of an eigenvalue at $\om$. The exponent $\alpha = 1/2$ in (\ref{L2est}) corresponds, in some sense, to the Lebesgue
spectrum (this is made precise at zero for self-similar translation flows with periodic renormalization in \cite[Theorem 6.2]{BuSol2}). Whereas $\alpha <1/2$ is possible at some points, for $\sig_f$-a.e.\ $\om$, the infimum of $\alpha$, for which (\ref{L2est}) holds, is at least $1/2$. Indeed, (\ref{locest}) implies that
$\loden(\sig_f,\om) = \liminf_{r\to 0} \frac{\log\sig_f (B(\om,r))}{\log r} \ge 2(1-\alpha)$, and it is well-known that
$$
\dim_H(\sig_f) = \sup \{s:\ \loden(\sig_f,\om) \ge s\ \mbox{for} \ \sig_f\mbox{-a.e.}\ \om\} \le 1, 
$$
see e.g.\ \cite[Prop.\ 10.2]{Falc-book}.
}

\end{remark}

\subsection{Exponential sums corresponding to suspension flows}
Let $X_+$ be a one-sided Markov compactum with a Vershik ordering and BV-transformation $\Tf$.
For a strictly positive vector $\vec{s} = (s_1,\ldots,s_m)$ we define the roof function $\phi_{\vec{s}}$ to be  equal to $s_a$ on the cylinder set $X_a^+$, as in Section 2, and obtain the suspension flow 
 $(\Xx^{\vec{s}},h_t)$. 

Recall that for $\ov{\ek}\in X_+$ (minus  a countable exceptional set) we defined its horizontal sequence $\hf(\ov{\ek})=(x_n)_{n\in \Z}$, in such a way
that the BV-transformation intertwines the left shift. Similarly, we can associate to $(\ov{\ek},t) \in \Xxs$  a tiling of the line $\R$: a symbol $a$
corresponds to a closed line segment of length $s_a$ (labeled by $a$), and these line segments are ``strung together'' according to the 
symbolic sequence $\hf(\ov{\ek})$. The tile corresponding to $x_0$ should contain the origin at the distance $t$ from the left endpoint. This
defines a map $\wt{\hf}$ from $\Xxs$ to a ``tiling space,'' which intertwines the flow $h_\tau$ and the left shift by $\tau$.

Our goal is to obtain growth estimates for twisted Birkhoff integrals (\ref{twist1}).
We start with test functions depending only on the cylinder set $X_a$ and the height $t$. More precisely, given some functions
 $\psi_a\in C([0,s_a])$, $a\in \A$, let
$$
f = \sum_{a\in \A} c_a f_a,\ \ \ \mbox{with}\ \ f_a(\ov{\ek},t) = \One_{\Xx_a} \psi_a(t),\ \ \mbox{where}\ \ 
\Xx_a=X_a\times [0,s_a].
$$
For a word $v$ in the alphabet $\Ak$ denote by $\vec{\ell}(v)\in \Z^m$ its ``population vector'' whose $j$-th entry is the number of $j$'s in $v$, for $j\le m$. We will  need the
``tiling length'' of $v$ defined by 
\be \label{tilength}
|v|_{\vec{s}}:= \langle\vec{\ell}(v), \vec{s}\rangle.
\ee
The following property will be used frequently: for an arbitrary substitution $\xi$, $\vec{s}>0$, and word $U\in \Ak^+$ we have
\be \label{tilep}
|U|_{\Sf_\xi^t \vec{s}} = \langle \vec{\ell}(U),\Sf_\xi^t \vec{s}\rangle = \langle \Sf_\xi\vec{\ell}(U), \vec{s}\rangle = \langle \vec{\ell}(\xi(U)), \vec{s}\rangle = |\xi(U)|_{\vec{s}}.
\ee

For $v=v_0\ldots v_{N-1}\in \Ak^+$ 
\be \label{def-Phi3}
\Phi_a^{\vec{s}}(v,\om) = \sum_{j=0}^{N-1} \delta_{v_j,a} \exp(-2\pi i \om |v_0\ldots v_{j-1}|_{\vec{s}}),
\ee
where the term for $j=0$ is equal to one by convention.
Then a straightforward calculation shows
\be \label{SR1}
S_R^{(\ov{\ek},0)}(f_a,\om) = \widehat{\psi}_a(\om) \cdot {\Phi_a^{\vec{s}}(x[0,N-1],\om)}\ \ \ \mbox{for}\ \ R = \left|x[0,N-1]\right|_{\vec{s}},
\ee
where 
$(x_n)_{n\in \Z}=\hf(\ov{\ek})$.
Moreover, the horizontal sequence can be represented as a concatenation of long blocks of the form $\zeta^{[n]}(b),\ b\in \Ak$ (not necessarily starting at the 0-th position); therefore, estimates of twisted Birkhoff sums for an arbitrary sequence may be reduced to those for $\zeta^{[n]}(b)$, and for the latter the renormalization naturally leads to matrix products. This is what we do next.

\subsection{Setting up matrix products}
Observe that for any two words $u,v$ and the concatenated word $uv$ we have
 \be \label{eq-Phi}
 \Phi^{\vec{s}}_\a(uv,\om) = \Phi^{\vec{s}}_\a(u,\om) + e^{-2\pi i \om |u|_{\vec{s}}} \cdot\Phi^{\vec{s}}_\a(v,\om).
 \ee
Recalling (\ref{def-subs1}), we can write
$$\zeta^{[n]}(b)= \zeta^{[n-1]}(\zeta_n(b))=\zeta^{[n-1]}(u_1^{b,n})\ldots \zeta^{[n-1]}(u_{|\zeta_n(b)|}^{b,n}),\ n\ge 1,$$ 
where we use the convention $\zeta^{[0]}:= Id$. Hence
(\ref{eq-Phi}) implies for all $a,b\in \A$:
$$
\Phi^{\vec{s}}_a(\zeta^{[n]}(b),\om) = \sum_{j=1}^{|\zeta_n(b)|} \exp\left[-2\pi i \om \left(|\zeta^{[n-1]}(u_1^{b,n}\ldots u_{j-1}^{b,n})|_{\vec{s}}\right)\right] \Phi^{\vec{s}}_a(\zeta^{[n-1]}(u_j^{b,n}),\om),\ n\ge 1,
$$
(for $j=1$, the exponential reduces to $\exp(0)=1$ by definition).
For $n\ge 0$ consider the $m\times m$ matrix-function $\Pbi_n(\om)$ defined by
\be \label{def-Psi}
\Pbi^{\vec{s}}_n(\om) = \left(\begin{array}{ccc} \Phi_1^{\vec{s}} (\zeta^{[n]}(1),\om) & \ldots & \Phi_m^{\vec{s}} (\zeta^{[n]}(1),\om) \\
\ldots\ldots\ldots & \ldots & \ldots\ldots\ldots \\
\Phi_1^{\vec{s}} (\zeta^{[n]}(m),\om) & \ldots & \Phi_m^{\vec{s}} (\zeta^{[n]}(m),\om). \end{array} \right).
\ee
It follows that
\be \label{eq-matr1}
\Pbi^{\vec{s}}_n(\om)=\M^{\vec{s}}_{n}(\om) \Pbi^{\vec{s}}_{n-1}(\om),\ \ n\ge 1,
\ee
where $\M^{\vec{s}}_{n}(\om)$ is an $m\times m$ matrix-function, whose matrix elements  are trigonometric polynomials  given by
\be \label{matr}
(\M^{\vec{s}}_{n}(\om))(b,c) = \sum_{j \le |\zeta_n(b)|:\ u_j^{b,n} = c} \exp\left[-2\pi i \om \left(|\zeta^{[n-1]}(u_1^{b,n}\ldots u_{j-1}^{b,n})|_{\vec{s}}\right)\right],\ n\ge 1.
\ee
Note that $\M^{\vec{s}}_{n}(0) =\Sf_{n}^t$, the transpose of the $n$-th substitution matrix, for all $n\ge 1$. 
Since $\Pbi^{\vec{s}}_0(\om) = I$ (the identity matrix), it follows from (\ref{eq-matr1}) that
\be \label{def-Pbi}
\Pbi^{\vec{s}}_n(\om)= \M^{\vec{s}}_{n} (\om)\M^{\vec{s}}_{n-1}(\om)\cdots \M^{\vec{s}}_1(\om).
\ee 

There is another way to express the matrices $\M_n^{\vec{s}}(\om)$ which will be useful below. It follows from (\ref{def-Phi3}), (\ref{matr}), and (\ref{tilep}) that 
$$
\M_n^{\vec{s}}(\om)(b,c) = (\Phi_c)^{\Sf^t_{\zeta^{[n-1]}}\vec{s}}(\zeta_n(b),\om).
$$
This motivates the following definition for two arbitrary substitutions $\xi_1,\xi_2$:
\be \label{motiv0}
\M^{\vec{s}}_{\xi_1,\xi_2}(\om)(b,c) = (\Phi_c)^{\Sf_{\xi_1}^t \vec{s}}(\xi_2(b),\om),
\ee
so that 
\be \label{motiv}
\M^{\vec{s}}_n = \M^{\vec{s}}_{\zeta^{[n-1]},\zeta_n}.
\ee
A straightforward verification yields the following identity for arbitrary substitutions $\xi_1, \xi_2, \xi_3$:
\be \label{verif}
\M^{\vec{s}}_{\xi_1, \xi_2\xi_3} = \M^{\vec{s}}_{\xi_1\xi_2, \xi_3}\cdot \M^{\vec{s}}_{\xi_1,\xi_2}.
\ee

\subsection{Estimating matrix products}

\begin{defi} \label{goodret}
A word $v$ is called a {\bf good return word} for the substitution $\zeta$ if $v$ starts with some symbol $c$ (which can be any element of $\Ak$) and $vc$ occurs in the word $\zeta(b)$ for every $b\in \Ak$. Denote by $GR(\zeta)$ the set of good return words for $\zeta$.
\end{defi}

One of our main assumptions will be that a fixed substitution $\zeta$, with a strictly positive matrix $\Sf_\zeta=:Q$ and nonempty set of good return words, appears infinitely often in the sequence $\zeta_j$. Thus, it is convenient to perform ``telescoping'' of the Bratteli-Vershik diagram and assume from the start that every substitution $\zeta_j$ has the form
\be \label{canonic}
\zeta_j = \zeta \xi_j \zeta,
\ee
where $\xi_j$ is an arbitrary substitution.

Denote by $\one$ the vector of all $1$'s, and let ${\|x\|}_{\R/\Z}$ be the distance from $x\in \R$ to the nearest integer. 
For a strictly positive square $m\times n$ matrix $A$ let
$$
\col(A) = \max_{i,j,k} \frac{A_{ij}}{A_{kj}}.
$$
It is well-known and easy to check that if $Q$ is strictly positive $m\times m$ matrix and $A$ is any non-negative $m\times n$ matrix, then we have 
\be \label{eq-col}
\col(QA) \le \col(Q).
\ee

\begin{prop} \label{prop-Dioph2}
Let $X_+$ be a one-sided Markov compactum with a Vershik ordering, and let $\zeta_j$ be the corresponding sequence of substitutions, given by (\ref{def-subs1}).
Suppose that there exists a substitution $\zeta$ with a nonempty set of good return words, such that $Q=\Sf_\zeta$ is strictly positive and $\zeta_j = \zeta\xi_j \zeta$ for some substitution $\xi_j$ for all $j\ge 1$. Then there exists $c_1\in (0,1)$, depending only on the substitution $\zeta$, such that  for all $a,b\in \Ak$ and $N\in \N$, $\om \in [0,1)$,  and $\vec{s}>\vec{0}$,
\be \label{eq-Dioph}
|\Phi^{\vec{s}}_a(\zeta^{[N]}(b),\om)| \le {\|\Sf^{[N]}\|}_{1}\prod_{n\le N-1} \left( 1 - c_1\cdot\max_{v\in GR(\zeta)} {\|\om|\zeta^{[n]}(v)|_{\vec{s}}\|}^2_{\R/\Z}\right).
\ee
In fact, we can take 
\be \label{defc1}
c_1 = \bigl(2m \cdot \max_{i,j} Q_{ij} \cdot \col(Q^t)\bigr)^{-1}.
\ee 
\end{prop}

\begin{proof}
This is similar to the proof of \cite[Proposition 3.2]{BuSol2}, but there are a number of new technical details. Let $\vec{e}_a$ denote the unit basis vector in $\R^m$ corresponding to $a\in \Ak$.
In view of (\ref{def-Psi}), it suffices to estimate 
$\Pbi_N^{\vec{s}}(\om)\vec{e}_a$.
We will use the following notation: 
\begin{itemize} \item for  vectors $\vec{x},\vec{y}\in \R^m$, the inequality $\vec{x}\le \vec{y}$ means componentwise inequality, and similarly for real-valued matrices;
\item
the operation of taking absolute values of all entries for a vector $\vec{x}$  and a matrix $A$ will be denoted $\vec{x}^{\mathbf |\cdot|}$ and $A^{|\cdot|}$.
\end{itemize}
It is clear that for any, generally speaking, rectangular matrices $A,B$ such that the product $AB$ is well-defined, we have
\be \label{eq-useful}
(AB)^{|\cdot|}\le A^{|\cdot|} B^{|\cdot|}.
\ee 
We  fix $\om$ and $\vec{s}$ and omit them from  notation, so that $\M_{n}\equiv\M^{\vec{s}}_{n}(\om)$ and $\Pbi_n\equiv\Pbi_n^{\vec{s}}(\om)$.
Observe that (\ref{eq-useful}) and (\ref{eq-matr1}) imply for all $n\in \N$:
\be \label{eq-prod}
(\Pbi_n \vec{e}_a)^{|\cdot|} = (\M_{n}\cdots \M_1 \vec{e}_a)^{|\cdot|}  \le  \M_{n}^{|\cdot|}\cdots \M_1^{|\cdot|} \vec{e}_a \le \M_{n}^{|\cdot|}\cdots \M_1^{|\cdot|} \one.
\ee
In view of (\ref{motiv}), (\ref{verif}), and (\ref{canonic}), we have
\begin{eqnarray}
\M_n = \M_{\zeta^{[n-1]},\zeta_n} & = & \M_{\zeta^{[n-1]},\zeta\xi_n\zeta} \nonumber \\
& = & \M_{\zeta^{[n-1]}\zeta\xi_n,\zeta}\cdot \M_{\zeta^{[n-1]},\zeta\xi_n} \nonumber \\
& = & \M_{\zeta^{[n-1]}\zeta\xi_n,\zeta}\cdot \M_{\zeta^{[n-1]}\zeta,\xi_n}\cdot \M_{\zeta^{[n-1]},\zeta}. \label{cucum}
\end{eqnarray}
By the definition of a good return word,  for any $v\in GR(\zeta)$ and $b\in \A$, we can write 
\be \label{eq-ret}
\zeta(b) =  p^{(b)} vc q^{(b)},
\ee
where $p^{(b)}$ and $q^{(b)}$ are words, possibly empty, and $v$ starts with $c$. Let $\xi$ be any substitution on $\Ak$. First we are going to estimate the absolute value of 
$$
\M_{\xi,\zeta}(b,c)= (\Phi_c)^{\Sf^t_\xi\vec{s}}(\zeta(b),\om)
$$
from above. This is  a trigonometric polynomial with $\Sf_\zeta^t(b,c)$ exponential terms and all coefficients equal to one. By (\ref{def-Phi3}) and (\ref{eq-ret}),
this polynomial includes the terms
\begin{eqnarray*}
& & \exp\bigl(-2\pi i \om|p^{(b)}|_{\Sf_\xi^t\vec{s}}\bigr) + \exp\bigl(-2\pi i \om|p^{(b)}v|_{\Sf_\xi^t\vec{s}}\bigr) \\[1.1ex]
& = & \exp\bigl(-2\pi i \om|\xi(p^{(b)}\bigr)|_{\vec{s}}) + \exp\bigl(-2\pi i \om|\xi(p^{(b)})\xi(v)|_{\vec{s}}\bigr) \\[1.2ex]
& = & \exp\bigl(-2\pi i \om|\xi(p^{(b)})|_{\vec{s}}\bigr)\cdot \bigl( 1 + e^{-2\pi i \om|\xi(v)|_{\vec{s}}}\bigr).
\end{eqnarray*}
It follows that 
$$
|\M_{\xi,\zeta}(b,c)| \le \Sf_\zeta^t(b,c) - 2 +\bigl|1 + e^{-2\pi i \om|\xi(v)|_{\vec{s}}}\bigr|,
$$

\noindent and from the inequality
$
|1 + e^{2\pi i \tau} | \le 2 - (1/2){\|\tau\|}_{\R/\Z}^2\,,\ \ \tau\in \R,
$
we obtain
\be \label{eq-new11}
|\M_{\xi,\zeta}(b,c)|\le \Sf_\zeta^t(b,c) - \half {\bigl\|\om |\xi(v)|_{\vec{s}} \bigr\|}_{\R/\Z}^2.
\ee
Now, for an arbitrary $\vec{x} = [x_1,\ldots,x_{m}]^t >\vec{0}$ and $b\in \A$, using (\ref{eq-new11}) we can estimate
\begin{eqnarray}
(\M_{\xi,\zeta}^{|\cdot|}\vec{x})_b & = & \sum_{j=1}^{m}|\M_{\xi,\zeta}(b,j)|\cdot x_j \nonumber \\
                                & \le & \sum_{j=1}^{m}
                             \Sf_\zeta^t(b,j) x_j - \half {\bigl\|\om |\xi(v)|_{\vec{s}} \bigr\|}_{\R/\Z}^2\cdot x_c\nonumber \\
                                 & \le & \left(1-c_2\psi(\vec{x}) {\bigl\|\om |\xi(v)|_{\vec{s}} \bigr\|}_{\R/\Z}^2\right) \cdot \sum_{j=1}^{m} \Sf_\zeta^t(b,j) x_j \nonumber \\
                                  & =   & \left(1-c_2\psi(\vec{x}) {\bigl\|\om |\xi(v)|_{\vec{s}} \bigr\|}_{\R/\Z}^2\right) \cdot (\Sf_\zeta^t \vec{x})_b, \label{eq-new12}
\end{eqnarray}
where $$c_2 = \frac{1}{2m\max_{i,j} \Sf_\zeta^t(i,j)}=\frac{1}{2m\max_{i,j} Q_{ij}}\ \ \ \mbox{ and }\ \ \ 
\psi(\vec{x}) = \frac{\min_j x_j}{\max_j x_j}\,.
$$
Thus,
\be \label{au}
\M_{\xi,\zeta}^{|\cdot|}\vec{x}\le \left(1-c_2\psi(\vec{x}) {\bigl\|\om |\xi(v)|_{\vec{s}} \bigr\|}_{\R/\Z}^2\right) \cdot \Sf^t_\zeta\vec{x},
\ee
and $v\in GR(\zeta)$ is arbitrary.
We will apply the last inequality with $\xi=\zeta^{[n-1]}$
and
$$
\vec{x}=\vec{x}_n = (\Sf^{[n-1]})^t\one \in Q^t \R^m_+,
$$
recalling that $\Sf_{n-1} = \Sf_{\zeta_{n-1}} = Q\Sf_{\xi_{n-1}}Q$.
Since the matrix $Q$ is strictly positive, we have
$$
\psi(\vec{x}) = (\col(\vec{x}))^{-1} \ge (\col(Q^t))^{-1}\ \ \mbox{for any}\ \vec{x}\in Q^t \R^m_+
$$
by (\ref{eq-col}).
Note also that for any substitutions $\xi_1,\xi_2$ we have
$$
M_{\xi_1,\xi_2}^{|\cdot|} \le \Sf_{\xi_2}^t
$$
by the definition (\ref{motiv0}).
Therefore, taking (\ref{cucum}) and (\ref{au}) into account, we obtain
\begin{eqnarray*}
\M_n^{|\cdot|} (\Sf^{[n-1]})^t\one & \le & \M^{|\cdot|}_{\zeta^{[n-1]}\zeta\xi_n,\zeta}\cdot \M^{|\cdot|}_{\zeta^{[n-1]}\zeta,\xi_n}\cdot \M^{|\cdot|}_{\zeta^{[n-1]},\zeta} (\Sf^{[n-1]})^t\one \\
& \le & \Sf_\zeta^t \Sf^t_{\xi_n} \left(1-c_1 \max_{v\in GR(\zeta)}  {\bigl\|\om |\zeta^{[n-1]}(v)|_{\vec{s}} \bigr\|}_{\R/\Z}^2\right) \Sf_\zeta^t (\Sf^{[n-1]})^t\one \\
& =  & \left(1-c_1 \max_{v\in GR(\zeta)}  {\bigl\|\om |\zeta^{[n-1]}(v)|_{\vec{s}} \bigr\|}_{\R/\Z}^2\right) (\Sf^{[n]})^t \one.
\end{eqnarray*}
where $c_1$ is given by (\ref{defc1}). Iterating this inequality yields
$$
(\Pbi_n \vec{e}_a)^{|\cdot|} \le \prod_{n \le N-1} \left(1- c_1\max_{v\in GR(\zeta)} {\bigl\|\om |\zeta^{[n]}(v)| \bigr\|}_{\R/\Z}^2\right) \cdot(\Sf^{[N]})^t \one.
$$
Finally, note that the maximal component of $(\Sf^{[N]})^t \one$ is the maximal column sum of $\Sf^{[N]}$, which is $\|\Sf^{[N]}\|_1$,
and the proposition is proved completely. We emphasize that we used in an essential way that every $\zeta_n = \zeta \xi_n \zeta$ both starts and end with $\zeta$.
\end{proof}

\subsection{Cylindrical functions of higher order}
Suppose now that $f$ on $\Xx$ is a ``cylindrical function of level $\ell$'', that is, its value depends only on the first $\ell$ edges
of the path $\ov{\ek}$ and on the height $t$. It is then convenient to represent $h_\tau$ as a suspension flow with a different height function,
based on the decomposition (disjoint in measure)
$$
\Xx^{\vec{s}} = \bigcup_{a\in \A}  \Xx_a^{(\ell)},\ \ \ \mbox{where}\ \ \Xx^{(\ell)}_a= 
\{(\ov{\ek},t)\in \Xxs:\ \ov{\ek}\in X_+,\ x_\ell=F(\ek_{\ell+1})=a\}.
$$
The BV-transformation $\Tf$ ``changes'' a vertex $a$ at the $\ell$-th level after $h^{(\ell)}_a=|\zeta^{[\ell]}(a)|$ iterates. Thus, after we enter the cylinder $\Xx_a^{(\ell)}$, the flow $h_\tau$ stays in it for the time equal to 
\be \label{def-sl}
s_a^{(\ell)}:=|\zeta^{[\ell]}(a)|_{\vec{s}}= [(\Sf^{[\ell]})^t \vec{s}]_a.
\ee
More precisely, if $F(\ek_{\ell+1})=a$ and $(\ov{\ek},t)\in \Xxs$, then
\be \label{cyl1}
\exists\,t'\in\left[0,s_a^{(\ell)}\right]\ \ \ \mbox{such that}\ \ \ (\ov{\ek},t) = h_{t'}(\ek_1'\ldots\ek_\ell'\ek_{\ell+1} \ek_{\ell+2}\ldots,0),
\ee
where $\ek_1'\ldots\ek_\ell'$ is the minimal path in the Bratteli diagram from the vertex $a$ on level $\ell$ to the level $0$.
Observe that the horizontal sequence of a path $\ek_1'\ldots\ek_\ell'\ek_{\ell+1}\ek_{\ell+2}\ldots$, with $F(\ek_{\ell+1})=a$, begins with
$\zeta^{[\ell]}(a)$ and can be written as $\zeta^{[\ell]}(x^{(\ell)})$ for some $x^{(\ell)}\in \A^\N$. (In fact, $x^{(\ell)}=\hf(\sig^\ell\ek)\in \Yk^{(\ell)}$, see Section 2). To summarize this discussion, for any
real-valued continuous cylindrical function $f$ of level $\ell$ on $X_+$ there exist $c_a\in \R$ and $\psi_a^{(\ell)}\in C([0, s_a^{(\ell)}]),\ a\in \Ak$, such that
\be \label{cyl2}
f = \sum_{a\in \A} c_a f^{(\ell)}_a,\ \ \ \mbox{where}\ \ f_a^{(\ell)}(\ov{\ek},t) = \One_{\Xx^{(\ell)}_a} \psi^{(\ell)}_a(t'),\ \ \mbox{with}\ \ t'\ \ \mbox{from (\ref{cyl1})}.
\ee
Now we can also write down a generalization of (\ref{SR1}). Denote $\vec{s}^{(\ell)} = (s_a^{(\ell)})_{a\in \Ak}$ and assume
$\ov{\ek}' = \ek_1'\ldots\ek_\ell'\ek_{\ell+1}\ek_{\ell+2}\ldots$, with $F(\ek_{\ell+1})=a$. Then
\be \label{SR11}
S_R^{(\ov{\ek}',0)}(f^{(\ell)}_a,\om) = \widehat{\psi}^{(\ell)}_a(\om) \cdot {\Phi_a^{\vec{s}^{(\ell)}}\bigl(x^{(\ell)}[0,N-1],\om\bigr)}\ \ \ \mbox{for}\ \ R = \bigl|x^{(\ell)}[0,N-1]\bigr|_{\vec{s}^{(\ell)}},
\ee
where $(x_n)_{n\ge 0} = \zeta^{[\ell]}(x^{(\ell)}_n)_{n\ge 0}$ and  
$(x_n)_{n\in \Z}=\hf(\ov{\ek})$. In the next corollary we extend Proposition~\ref{prop-Dioph2} to cylindrical functions of level $\ell$. 

\begin{corollary} \label{cor-Dioph22}
Under the assumptions of Proposition~\ref{prop-Dioph2}, for any $\ell \ge 1$,  
$a,b\in \A$, $n\ge \ell+1$, $\vec{s}>\vec{0}$, and $\om\in \R$, we have
\be \label{eq-new191}
\bigl|\Phi_a^{\vec{s}^{(\ell)}}(\zeta^{[\ell+1,n]}(b),\om)\bigr| \le  {\|\Sf^{[\ell+1,n]}\|_1}\ \cdot \!\!\!\!\!\!\!\prod_{\ell+1\le k\le n-1} \!\!
\left( 1 - c_1\cdot\!\!\!\max_{v\in GR(\zeta)} {\|\om|\zeta^{[k]}(v)|_{\vec{s}}\|}_{\R/\Z}^2\right), 
\ee
where $c_1$ is given by (\ref{defc1}).
\end{corollary}

\begin{proof} It is immediate from 
Proposition~\ref{prop-Dioph2}  by shifting the indices that
\begin{eqnarray*}
\bigl|\Phi_a^{\vec{s}^{(\ell)}}(\zeta^{[\ell+1,n]}(b),\om)\bigr| & \le &  {\|\Sf^{[\ell+1,n]}\|_1}\ \cdot \!\!\!\!\!\!\!\prod_{\ell+1\le k\le n-1} \!\!
\left( 1 - c_1\cdot\!\!\!\max_{v\in GR(\zeta)} {\|\om|\zeta^{[\ell+1,k]}(v)|_{\vec{s}^{(\ell)}}\|}_{\R/\Z}^2\right). 
\end{eqnarray*}
It remains to note  that
\begin{eqnarray*}
|\zeta^{[\ell+1,k]}(v)|_{\vec{s}^{(\ell)}} & = & \langle \vec{\ell}(\zeta^{[\ell+1,k]}(v)), \vec{s}^{(\ell)}\rangle \\
                                                              & = & \langle \Sf^{[\ell+1, k]} \vec{\ell}(v), (\Sf^\ell)^t \vec{s}\rangle \\
& = & \langle \Sf^\ell \Sf^{[\ell+1,k]} \vec{\ell}(v),  \vec{s}\rangle \\
& = & \langle  \Sf^{[k]} \vec{\ell}(v),  \vec{s}\rangle \\
& = & \langle \vec{\ell}(\zeta^{[k]}(v)), \vec{s}\rangle 
 = |\zeta^{[k]}(v)|_{\vec{s}}.
\end{eqnarray*}
\end{proof}

Next we need to pass from the exponential sum corresponding to the word $\zeta^{[\ell+1,n]}(b)$ to the one corresponding to a general word in the space $\Yk^{(\ell)}$.
To this end, we will use  the well-known prefix-suffix decomposition.

\begin{lemma} \label{lem-accord} Let $x^{(\ell)}\in \Yk^{(\ell)}$ and $N\geq 1$. Then
 \begin{equation} \label{eq-accord}
x^{(\ell)} [0,N-1]=\zeta^{[\ell+1]}(u_{\ell+1})\zeta^{[\ell+1,\ell+2]}(u_{\ell+2})\ldots \zeta^{[\ell+1,n]}(u_{\ell+n})\zeta^{[\ell+1,n]}(v_{\ell+n}) \ldots \zeta^{[\ell+1]}(v_{\ell+1}),
\end{equation}
where  $u_j,v_j,\ j=\ell+1,\ldots,\ell+n$, are respectively proper suffixes and prefixes  of the words $\zeta_{j+1}(b)$, $b\in \mathcal A$. The words $u_j, v_j$ may be empty, except that  at least one of $u_{\ell+n}, v_{\ell+n}$ is
nonempty. Moreover,
\be \label{Nn-cond}
\min_{b\in \Ak}|\zeta^{[\ell+1,n]}(b)| \le N \le 2\max_{b\in \Ak} |\zeta^{[\ell+1,{n+1}]}(b)|.
\ee
\end{lemma}

\begin{proof} This is immediate from the description of $\Yk^{(\ell)}$ at the end of Section 2.
\end{proof}

\begin{prop}\label{prop-Dioph3}
Under the assumptions of Proposition~\ref{prop-Dioph2}, for any $\ell\ge 1$, 
$a \in \A$, $N\in \Nat$, $\vec{s}>\vec{0}$, $x^{(\ell)}\in \Yk^{(\ell)}$, and $\om\in \R$, we have, 
\be \label{eq-new192}
\bigl|{\Phi_a^{\vec{s}^{(\ell)}}\bigl(x^{(\ell)}[0,N-1],\om\bigr)}\bigr| \le  2\sum_{j=\ell}^n \|\Sf^{[\ell+1,j]}\|_1\cdot \|\Sf_{j+1}\|_1\ \cdot\!\!\!\!\!\!\!\!
\prod_{\ell+1\le k\le j-1}\!\! \!\!\bigl( 1 - c_1\cdot\!\!\max_{v\in GR(\zeta)} {\|\om|\zeta^{[k]}(v)|_{\vec{s}}\|}_{\R/\Z}^2\bigr), 
\ee
where $c_1$ is given by (\ref{defc1}) and $n\in \N$ is such that (\ref{Nn-cond}) holds.
Here we let $\Sf^{[\ell+1,\ell]}=:I$. 
\end{prop}

\begin{proof}
We use Lemma~\ref{lem-accord}, and apply Corollary~\ref{cor-Dioph22}  to each term. The factor 
$2\|\Sf_{j+1}\|_1$
 in (\ref{eq-new192}) appears, because $|u_j|, |v_j| \le \max_b |\zeta_j(b)| = \|\Sf_{j+1}\|_1$ in (\ref{eq-accord}).
\end{proof}


\section{Random BV-transformations: statement of the theorem and plan of the proof}

Here we consider dynamical systems generated by a {\em random} sequence of Markov compacta. In order to state our results, we need some preparation; specifically, the Oseledets Theorem.

Recall that $\Frg$ denotes the set of all oriented graphs on $m$ vertices such that there is an edge starting at every vertex and an edge ending at every vertex (we allow loops and multiple edges). We also assume that each graph is equipped with a Vershik ordering. Let $\Om$ be the space of sequences of graphs:
$$
\Om = \{\bom = \ldots \bom_{-n}\ldots\bom_0{\mbox{\bf .}}\bom_1\ldots \bom_n\ldots,\ \bom_i \in \Frg,\ i\in \Z\}.
$$
For $\bom\in\Om$ we denote by $X(\bom)$  the Markov compactum corresponding to $\bom$ according to the rule
$
\Gam_n = \bom_{n},\ \ n\in \Z,
$
and let $\sig$ the left shift on $\Om$. We also consider the corresponding one-sided compactum $X_+(\bom)$. 
For a word $\bq = \bq_1\ldots \bq_k\in \Frg^k$ we can ``concatenate'' the graphs to obtain the ``aggregated'' graph $\Gam_\bq$, also belonging to $\Frg$. By the definition of incidence matrix, we  have
$$
A(\bq) := A(\Gam_\bq) = A(\bq_k)\cdot\ldots\cdot A(\bq_1).
$$
Since the graphs are equipped with the Vershik ordering, we also have a corresponding sequence of substitutions, so that $\zeta(\bq) = \zeta(\bq_1)\ldots \zeta(\bq_k)$. We will also need a ``2-sided cylinder set'':
$$
[\bq.\bq] = \{\bom\in \Om:\ \bom_{-k+1}\ldots\bom_0 = \bom_1\ldots\bom_k = \bq\}.
$$
Following \cite{BufGur}, we say that the word $\bq = q_1\ldots q_k$ is ``simple'' if for all $2 \le i \le k$ we have $q_i\ldots q_k \ne q_1\ldots q_{k-i+1}$. If the word $\bq$ is simple, two occurrences of $\bq$ in the sequence
$\bom$ cannot overlap. 
Let $\P$ be an ergodic $\sig$-invariant probability measure on $\Om$ satisfying the following 

\medskip

\noindent
{\bf Conditions:} 

{\bf (C1)} {\em  There exists a word $\bq\in \Frg^k$ such that all the entries of the matrix $A(\bq)$ are positive and
\begin{equation}\label{gamone}
\P(\bqcyl)>0.
\end{equation}

{\bf (C2)} The matrices $A(\bom_n)$ are almost surely invertible with respect to $\P$.

{\bf (C3)} The functions $\bom\mapsto \log(1+ \|A^{\pm 1}(\bom_1)\|)$ are integrable.} 

\noindent (Here and below
$\|A\|$ denotes the Euclidean operator norm of the matrix.)

\medskip

Observe that {\bf (C3)}, together with the Birkhoff ergodic theorem, immediately gives
\be \label{immed1}
\lim_{n\to \infty} n^{-1}\log(1 + \|A(\bom_n)\|) = 0\ \ \mbox{for $\P$-a.e.}\ \bom\in \Om. 
\ee
We obtain a measurable cocycle $\AA: \Om\to GL(m,\R)$, defined by $\AA(\bom) = A(\bom_1)$, called the {\em renormalization cocycle}. Denote
\begin{equation}\label{renormcoc}
\AA(n,\bom) = \left\{ \begin{array}{lr} \AA(\sig^{n-1}\bom)\cdot \ldots \cdot \AA(\bom), & n> 0; \\
                                                     Id, & n=0; \\
                                                      \AA^{-1}(\sig^{-n}\bom) \cdot \ldots \cdot \AA^{-1}(\sig^{-1}\bom), & n<0,\end{array} \right.
\end{equation}
so that 
$$\AA(n,\bom) =A(\bom_{n})\cdots A(\bom_1),\ n\ge 1.$$
As in Section 2, we 
consider the sequence of substitutions
$
\zeta(\bom_{k}),\ \bom\in \Om,\ k\in \Z,
$
and their substitution matrices $\Sf_{\zeta_k(\bom)}=A^t(\bom_k)$. (Recall that all graphs $\bom_k$ are equipped with a Vershik ordering.) Thus
$$\AA(n,\bom) = \Sf^t(\bom_n\ldots \bom_1)= \Sf^t_{\zeta(\bom_1\ldots\bom_n)}\ n\ge 1.$$
By the Oseledets Theorem \cite{oseledets} (for a detailed survey, see Barreira-Pesin \cite{barpes}),
there exist Lyapunov exponents $\theta_1> \theta_2 > \ldots > \theta_r$ and, for $\P$-a.e.\ $\bom\in \Om$, a direct-sum decomposition
\be \label{os1}
\R^m = E^1_\bom \oplus \cdots \oplus E^r_\bom 
\ee
that depends measurably on $\bom\in \Om$ and satisfies the following:

(i) for $\P$-a.e.\ $\bom\in \Om$, any $n\in\Z$, and any $i=1,\ldots,r$ we have
$$
\AA(n,\bom) E^i_\bom = E^i_{\sig^n\bom};
$$

(ii) for any $v\in E^i_\bom,\ v\ne 0$, we have
$$
\lim_{|n|\to \infty} \frac{\log\|\AA(n,\bom)v\|}{n} = \theta_i.
$$

(iii) $\lim_{|n|\to \infty} \frac{1}{n}\log\angle\left(\bigoplus_{i\in I} E^i_{\sig^n\bom}, \bigoplus_{j\in J} E^j_{\sig^n \bom}\right) =0$ whenever $I\cap J = \es$.




Let $P^i_\bom$ be the projection to $E^i_\bom$ arising from (\ref{os1}).
Denote  by $\sig_f$ the spectral measure for the system $(\Xx^{\vec{s}},h_t)$ with the test function $f$ (assuming the system is uniquely ergodic).
Now we can state our theorem.

\begin{theorem} \label{th-main1}
Let $(\Om,\P,\sig)$ be an invertible ergodic measure-preserving system satisfying conditions {\bf (C1)-(C3)} above. Consider the cocycle $\AA(n,\bom)$ defined by (\ref{renormcoc}).  Assume that
\begin{enumerate}
\item[(a)] the Lyapunov spectrum satisfies
$$
\theta_1 > \theta_2 > 0 > \theta_3 > \ldots,
$$
and the two top exponents are simple (i.e.\ $\dim(E_\bom^1) = \dim(E^2_\bom)=1$ for $\P$-a.e.\ $\bom$);
\item[(b)] there exists a simple word $\bq\in \Frg^k$ for some $k\in \N$, such that all the entries of the matrix $A(\bq)$ are strictly positive and $\P(\bqcyl)>0$;
\item[(c)] there exist ``good return words'' $\{u_j\}_{j=1}^m$ for $\zeta=\zeta(\bq)$ (see Definition~\ref{goodret}), such that $\{\vec{\ell}(u_j)\}_{j=1}^m$ is a
basis for $\R^m$;
\item[(d)] Let $\ell_\bq(\bom)$ be the ``negative'' waiting time until the first appearance of $\bq.\bq$, i.e.
$$
\ell_\bq(\bom) = \min\{n\ge 1:\, \sig^{-n}\bom \in \bqcyl\}.
$$
Let $\P(\bom|\bom^+)$ be the conditional distribution on the set of $\bom$'s
conditioned on the future $\bom^+ = \bom_1\bom_2\ldots$ We assume that there exist $\eps>0$ and $1<C<\infty$ such that
\be \label{eq-star1}
\int_{\bqcyl} \left\|\AA(\ell_\bq(\bom), \sig^{-\ell_\bq(\bom)}\bom)\right\|^\eps\,d\P(\bom|\bom^+) \le C\ \ \ \mbox{for all}\ \ \bom^+\  \mbox{starting with}\ \q.
\ee
\end{enumerate}
Then there exists $\gam>0$  such that for $\P$-a.e.\ $\bom \in \Om$ the following holds:

Let $(X_+,\Tf)$ be 
the Bratteli-Vershik system corresponding to $\bom^+$, which is uniquely ergodic.
 Let $(\Xx^{\vec{s}},h_t)$ be the suspension flow over $(X_+,\Tf)$ under the piecewise-constant roof function determined by $\vec{s}$. Then for 
all $\beta>0$ and $B>1$ there exists $r_0=r_0(\bom,\beta,B)>0$, such that 
for Lebesgue-a.e.\ $\vec{s}$, with $\|\vec{s}\|_1=1$ and $\min_{j=1,2}|P^j_\bom(\vec{s})|\ge \beta$,  for any $f\in \Lip_w^+(\Xx^{\vec{s}})$,
\be \label{main-Hoeld}
\sig_f(B(\om,r))\le C(\bom,\|f\|_L)\cdot r^\gam\ \ \ \mbox{for all}\ \ \om\in [B^{-1},B]\ \ \mbox{and}\ \ 0<r< r_0,
\ee
with the constant depending only on $\bom$ and $\|f\|_L$.
\end{theorem} 

\noindent {\bf Remarks.}  1. It is clear that condition {\bf (C1)} follows from assumption (b), but we chose to state {\bf (C1)} explicitly, since this is the condition which appears in the literature and implies unique ergodicity. The unique ergodicity of the system $(X_+,\Tf)$ for $\P$-a.e.\ $\bom$ under the given assumptions
is well-known and goes back to the work of Furstenberg \cite{furst} (see the beginning of Section 2).

2. The assumption that $\bq$ is a simple word ensures that occurrences of $\bq$ do not overlap. Then we have
\be \label{egstar}
\AA(\ell_\bq(\bom), \sig^{-\ell_\bq(\bom)}\bom) = A(\bq) A(\bp) A(\bq),
\ee
for some $\bp \in \Frg$ (possibly trivial). For our application, it will be easy to make sure that $\bq$ is simple, as we show in Section 11, unlike in the paper \cite{BufGur}, where additional efforts were needed to achieve the desired aims.

\medskip

The scheme of the proof is as follows: first we reduce the theorem to the case where all the symbols  have the form $\bom_n = \bq\bp_n\bq$. This is done by considering the first return map to the cylinder set
$\bqcyl$.  Then we apply  Proposition~\ref{prop-Dioph3}, with the goal to use Lemma~\ref{lem-varr}.
In order to achieve the desired estimate, roughly speaking, we need to show that for $\P$-a.e. sequence of substitutions, for Lebesgue a.e.\ $\vec{s}$,
the distance from $\om|\zeta^{[n]}(v_n)|_{\vec{s}}$ to the nearest integer (for some choice of a good return word $v_n$ which  depends on
$n$)  is bounded away from zero for a positive frequency of $n$'s (uniformly in $\om$  bounded away from zero and infinity). The proof splits into two parts, separating the two ``almost every''.
The first part is probabilistic, showing that certain assumptions on the sequence of substitutions $\zeta(\bp_n)$ hold $\P$-almost surely.
In the second part we fix a typical sequence $\zeta(\bp_n)$ and obtain estimates for a.e.\ $\vec{s}$. This is done using the
``Erd\H{o}s-Kahane argument.''


\section{Reduction}

In this short section, we show that Theorem~\ref{th-main1} reduces to the case when 
\be \label{induce}
\bom_n = \bq \bp_n \bq\ \ \ \mbox{for all}\ \ n\in \Z,
\ee
where $\bq$ is a fixed graph with fixed Vershik ordering, such that its incidence matrix  is strictly positive, and $\bp_n$ is arbitrary. In the next theorem we use the same notation as in Theorem~\ref{th-main1}.

\begin{theorem} \label{th-main11}
Let $(\Om_\bq,\P,\sig)$ be an invertible ergodic measure-preserving system of the form (\ref{induce}) satisfying conditions {\bf (C1)-(C3)} from Section 4. Consider the cocycle $\AA(n,\bom)$ defined by (\ref{renormcoc}).  Assume that
\begin{enumerate}
\item[(a$'$)] the Lyapunov spectrum satisfies
$$
\theta_1 > \theta_2 > 0 > \theta_3 > \ldots,
$$
and the two top exponents are simple;
\item[(b$'$)] the substitution $\zeta = \zeta(\bq)$ is such that its substitution matrix $\Sf_\zeta=Q$ has strictly positive entries;
\item[(c$'$)] there exist ``good return words'' $\{u_j\}_{j=1}^m$ for $\zeta=\zeta(\bq)$ (see Definition~\ref{goodret}), such that $\{\vec{\ell}(u_j)\}_{j=1}^m$ is a
basis for $\R^m$;
\item[(d$'$)]  there exist $\eps>0$ and $1<C<\infty$ such that
\be \label{eq-star11}
\int_{\Om_\bq} \left\|A(\bom_0)\right\|^\eps\,d\P(\bom|\bom^+) \le C\ \ \ \mbox{for all}\ \ \bom^+.
\ee
\end{enumerate}
Then there exists $\gam>0$  such that for $\P$-a.e.\ $\bom \in \Om_\bq$ the following holds:

Let $(X_+,\Tf)$ be 
the Bratteli-Vershik system corresponding to $\bom^+$, which is uniquely ergodic.
 Let $(\Xx^{\vec{s}},h_t)$ be the suspension flow over $(X_+,\Tf)$ under the piecewise-constant roof function determined by $\vec{s}$. Then for 
all $\beta>0$ and $B>1$ there exists $r_0=r_0(\bom,\beta,B)>0$, such that 
for Lebesgue-a.e.\ $\vec{s}$, with $\|\vec{s}\|_1=1$ and $\min_{j=1,2}|P^j_\bom(\vec{s})|\ge \beta$,  for any $f\in \Lip_w^+(\Xx^{\vec{s}})$,
\be \label{main-Hoeld}
\sig_f(B(\om,r))\le C(\bom,\|f\|_L)\cdot r^\gam\ \ \ \mbox{for all}\ \ \om\in [B^{-1},B]\ \ \mbox{and}\ \ 0<r< r_0,
\ee
with the constant depending only on $\bom$ and $\|f\|_L$.
\end{theorem} 

\noindent {\bf Remark.}
 If we assume that $\P$ is ``quasi-Bernoulli'', i.e.\ it satisfies the ``bounded distortion property'' of \cite{AV},  then (\ref{eq-star11}) can be replaced by the ``unconditional'' estimate $\int_{\Om_\bq} \left\|A(\bom_0)\right\|^\eps\,d\P(\bom) \le C$. However, we prefer the current formulation.
 
\begin{proof}[Proof of Theorem~\ref{th-main1} assuming Theorem~\ref{th-main11}] Given an ergodic system $(\Om,\P,\sig)$ from the statement of Theorem~\ref{th-main1}, we consider the induced system on the cylinder set
$\Om_\bq:= [\bq.\bq]$. 
Then symbolically we can represent elements of $\Om_\bq$ as sequences satisfying (\ref{induce}). Denote by $\P_{\!\!\bq}$ the induced (conditional) measure on $\Om_\bq$. Since $\P([\bq.\bq])>0$, standard results in Ergodic Theory imply that the resulting induced system $(\Om_\bq,\P_{\!\!\bq},\sig)$ is also ergodic and the associated cocycle has the same properties of the Lyapunov spectrum (with the values of the Lyapunov exponents multiplied  by $1/\P(\bqcyl)$); that is, (a$'$) holds.The properties (b$'$) and (c$'$) follow from (b) and (c) automatically. Finally, note that (\ref{eq-star11}) is identical to (\ref{eq-star1}). On the level of Bratteli-Vershik diagrams this corresponds to the ``aggregation-telescoping procedure'', which results in a naturally isomorphic Bratteli-Vershik system.  Observe also that a weakly-Lipschitz function on $\Om$ induces a weakly-Lipschitz function on $\Om_\bq$ without increase of the norm $\|f\|_L$, see Section 2.1. Thus,  Theorem~\ref{th-main11} applies, and the reduction is complete.
\end{proof}

The next five sections are devoted to the proof of Theorem~\ref{th-main11}. 
In the last Section 11 we return to the setting of Theorem~\ref{th-main1} and deduce Theorem~\ref{main-moduli} from it.


\section{Exponential tails}

For $\bom\in  \Om_\bq$ we consider the sequence of substitutions $\zeta(\bom_n)$, $n\ge 1$. In view of (\ref{induce}), we have
$$
\zeta(\bom_n) =  \zeta(\bq) \zeta(\bp_n)\zeta(\bq).
$$
Recall that 
$$
A(\bp_n) = \Sf_{\zeta(\bp_n)}^t.
$$
Denote
\be \label{def-W}
W_n = W_n(\bom):=\log\|A(\bom_n)\| = \log\|Q^t A(\bp_n) Q^t\|.
\ee
Since all the matrices in the product have non-negative integer entries and a.s.\ invertible, and the first and last one are equal to $Q$ with strictly positive entries, we have $W_n>0,\ n\ge 1,$ for $\P$-a.e.\ $\bom$.  This will always be assumed below,
without loss of generality.

\begin{prop} \label{prop-proba}
Under the assumptions of Theorem~\ref{th-main11},
there exists a positive constant $L_1$  such that for $\P$-a.e.\ $\bom$, the following holds: for any $\delta>0$, for all $N$ sufficiently large ($N\ge N_0(\bom,\delta)$),
\be \label{w-cond2}
\max\left\{\sum_{n\in \Psi} W_n:\ \Psi\subset \ \{1,\ldots,N\},\  |\Psi| \le \delta N\right\} \le L_1\cdot \log(1/\delta)
\cdot \delta N.
\ee
\end{prop}

We will prove the proposition at the end of the section, but first point out the following.

\begin{remark} \label{pointout}
It follows from (\ref{w-cond2})  that for any $\wtil{\delta}>0$, for all $n$ sufficiently large,
\be \label{eqrem}
W_n \le \wtil{\delta} n.
\ee
Indeed, in (\ref{w-cond2}) we just need to take $\delta>0$ such that $L_1\cdot\log(1/\delta)\cdot \delta < \wtil{\delta}$, and then $\Psi = \{n\}$, which
clearly satisfies the condition $1=|\Psi|\le \delta N$ for $N$ sufficiently large. 

As the referee pointed out, this also follows directly from the Birkhoff Ergodic Theorem, since $\frac{1}{n}(W_1 + \cdots + W_n) \to \int_{\Om_{\bq}}\log\|A(\bom)\|\,d\P<\infty$ for a.e.\ $\bom\in \Om_\bq$.
\end{remark}

\begin{lemma} \label{lem3}
We have for all $N$ and $n$, and for any (deterministic!) increasing sequence $1\le j_1 < j_2 < \ldots < j_n$:
\be \label{eq-devi2}
\P\left[\sum_{i=1}^n {W}_{j_i} \ge Kn\right] \le \exp(-\eps K n/2)\ \  \mbox{for} \ \ K\ge \frac{2\log C}{\eps}\,,
\ee 
where $\eps>0$ and $C>1$ are the constants from (\ref{eq-star1}).
\end{lemma}

This is a standard large deviation result, but we provide a proof for completeness. Thanks to Chris Hoffman who showed us the argument.

\begin{proof}
Let $X_i = W_{j_i} - K$. Then $\P\left[\sum_{i=1}^n W_{j_i} \ge Kn\right]=
\P\bigl[\sum_{i=1}^n X_i \ge 0\bigr]$. 
Observe that the values of $W_{j_2}$, \ldots $W_{j_n}$ are determined by the ``future'' of $\sig \bom $, that is, by $(\sig\bom)^+= \bom_2\bom_3\ldots$, hence by (\ref{eq-star1}) we have
$$
\E\left[e^{\eps W_{j_1}}\mid W_{j_2},\ldots, W_{j_n}\right] < C.
$$
Therefore,
\be \label{eq-devi3}
\E\left[e^{\eps X_1}\mid X_{2},\ldots,X_n\right] < C e^{-\eps K} \le e^{-\eps K/2},
\ee
provided $K\ge 2\eps^{-1}\log C$.
Let $S_\ell = \sum_{i=n-\ell+1}^n X_i$. Now,
\begin{eqnarray*}
\E\left[e^{\eps S_n}\right] & = & \sum_b \E\left[e^{\eps S_n}\mid e^{\eps S_{n-1}}=b\right]\cdot 
\P\left[e^{\eps S_{n-1}}=b\right] \\
& = & \sum_b b\cdot \E\left[e^{\eps X_n}\mid e^{\eps S_{n-1}}=b\right]\cdot 
\P\left[e^{\eps S_{n-1}}=b\right] \\
& \le & e^{-\eps K/2} \sum_b b\cdot \P\left[e^{\eps S_{n-1}}=b\right] = e^{-\eps K/2} 
\E\left[e^{\eps S_{n-1}} \right],
\end{eqnarray*}
taking (\ref{eq-devi3}) into account. Iterating the last inequality yields
$$
\E\left[e^{\eps S_n}\right] \le e^{-\eps Kn/2},
$$
and since $\P[S_n\ge 0] \le \E\left[e^{\eps S_n}\right]$, the estimate (\ref{eq-devi2}) is proved.
\end{proof}

\begin{proof}[Proof of Proposition~\ref{prop-proba}] Consider the event 
$$
\Wk(N,\delta,K) = \left\{\max_{\stackrel{\scriptstyle{\Psi\subset \{1,\ldots,N\}}}{|\Psi|\le\delta N}} \sum_{n\in \Psi} W_n \ge K(\delta N)\right\}
$$
Then we have for $K\ge 2\log C/\eps$, 
\begin{eqnarray*}
\P\left(\Wk(N,\delta,K)\right) 
& \le & \!\!\!\!\sum_{\stackrel{\scriptstyle{\Psi\subset \{1,\ldots,N\}}}{|\Psi|\le\delta N}}  \P\left[ \sum_{n\in \Psi} {W}_n \ge K(\delta N)\right] \\[1.2ex] 
& \le & \sum_{i\le \delta N} {N\choose i}e^{-\eps K(\delta N)/2},
\end{eqnarray*}
in view of Lemma~\ref{lem3}.
By Stirling,  there exists $C'>1$ such that 
\be \label{Stirling}
\sum_{i\le \delta N} {N\choose i}\le \exp\left[C' \delta \log(1/\delta)N\right]\ \ \mbox{for}\ \  \delta<e^{-1}\ \ \mbox{and all}\ \  N>1.
\ee
Therefore,
$$
 \sum_{i\le \delta N} {N \choose i} \le \exp[-\eps K (\delta N)/4] \ \ \mbox{for}\ \ K = \frac{4C'}{\eps} \log(1/\delta),
$$
whence, by Borel-Cantelli, the event $\Wk(N,\delta, L_1 \log(1/\delta))$ does not occur for all $N$ sufficiently large, with
$$
L_1= \eps^{-1} \max(4C', 2\log C),
$$
which means that condition (\ref{w-cond2}) holds.
\end{proof}


\section{Estimating twisted Birkhoff integrals}

In this section we continue to work with a $\P$-generic 2-sided sequence $\bom\in \Om_\bq$.
Under the assumptions of Theorem~\ref{th-main11}, for $\P$-a.e.\ $\bom$, the sequence of substitutions $\zeta(\bom_n)$, 
$n\in \Z$, satisfies several conditions.
First of all, we can assume that the point $\bom$ is generic for the Oseledets Theorem; that is, assertions (i)-(iii) from Section 4 hold.  We further assume that the conclusions
of Proposition~\ref{prop-proba} hold.  Recall that  $\zeta(\bom_n) = \zeta(\bq)\xi_n \zeta(\bq)$, where $Q = \Sf_\zeta$ is a strictly positive matrix. Below we denote by $O_Q(1)$ a generic constant which depends only on $Q = \Sf(\zeta)$ and which may be different from line to line.


\begin{prop} \label{prop-Dioph4} Suppose that the conditions of Theorem~\ref{th-main11} are satisfied. Then
for $\P$-a.e.\ $\bom\in \Om$, 
 for any $\eta\in (0,1)$, there exists $\ell_\eta=\ell_\eta(\bom)\in \N$, such that 
 for all $\ell\ge \ell_\eta$ and any bounded cylindrical function $f^{(\ell)}$ of level $\ell$, 
 for any $(\ov{\ek},t)\in \Xxs$, with
$\ov{\ek}\in X_+(\bom)$, and $\om\in \R$, 
\be \label{eq-Dioph3}
|S^{(\ov{\ek},t)}_R(f^{(\ell)},\om)| \le  O_Q(1)\cdot \|f^{(\ell)}\|_{_\infty} \Bigl(R^{1/2}+ R^{1+\eta}\!\!\!\!\!\prod_{\ell+1\le k < \frac{\log R}{4\theta_1}} 
\Bigl( 1 - c_1\cdot \!\!\!\!   \max_{v\in GR(\zeta)}\bigl\| \om|\zeta^{[k]}(v)|_{\vec{s}}\bigr\|^2_{\R/\Z}\Bigr)\Bigr),
\ee
for all
$
R\ge e^{8\theta_1 \ell}.
$
\end{prop}

\begin{remark} \label{rem-return} {\em 
By (\ref{tilep}) we have
$$
\|\om|\zeta^{[n]}(v)|_{\vec{s}}\|_{\R/\Z} = \|\langle \vec\ell(v), \om (\Sf^{[n]})^t \vec{s} \rangle \|_{\R/\Z} = \|\langle \vec\ell(v), \AA(n,\bom) (\om\vec{s} )\rangle \|_{\R/\Z}.
$$
In fact, our assumption (namely, condition (c$'$) in Theorem~\ref{th-main11}) is that the substitution $\zeta$ possesses $m$ good return words $v_1,\ldots,v_m$ such that their population vectors $\vec\ell(v_1), \ldots \vec\ell(v_m)$ form a basis of $\R^m$. Observe that $\langle \vec{\ell}(v_j), \vec{x}\rangle$, for $j=1,\ldots,m$, are the coordinates of a vector $\vec{x}\in \R^m$ with respect to the basis dual to $\{\vec\ell(v_1), \ldots \vec\ell(v_m)\}$.
Let $\Gamma$ be the free Abelian group generated by $\vec\ell(v_1), \ldots \vec\ell(v_m)$. Then $\Gamma < \Z^m$ is a full rank lattice. Let $\widehat{\Gam}$ be the dual lattice. Observe that
$$
C_\zeta^{-1} \|\vec x\|_{\R^m/\widehat \Gam} \le \max_{j\le m} \|\langle \vec\ell(v_j), \vec{x}\rangle \|_{\R/\Z} \le C_\zeta \|\vec x\|_{\R^m/\widehat \Gam},
$$
with $C_\zeta>1$ depending only on $\vec\ell(v_1), \ldots \vec\ell(v_m)$. Thus, estimating the product in (\ref{eq-Dioph3}) is equivalent to estimating
\be \label{lattice}
\prod_{\ell+1 \le k < \frac{\log R}{4\theta_1}} \left( 1 - \wt c_1\cdot \bigl\|\AA(k,\bom) (\om\vec{s})\bigr\|^2_{\R^m/\widehat \Gam} \right),
\ee
making it similar to the expression appearing in the Veech criterion alluded to in Section 1.3. However, although the form of (\ref{lattice}) may be more appealing, for technical reasons we prefer to work with the expression in (\ref{eq-Dioph3}).
}
\end{remark}

\begin{proof}[Proof of Proposition~\ref{prop-Dioph4}]
Without loss of generality we can assume that $f^{(\ell)} = f^{(\ell)}_a$ for some $a\in \Ak$, as in (\ref{cyl2}) and find $\ov{\ek}'$ and $t'$ as in (\ref{cyl1}). Since $(\ov{\ek},t) = h_{t'}(\ov{\ek}',0)$, we have
\begin{eqnarray*}
|S^{(\ov{\ek},t)}_R(f_a^{(\ell)},\om)| & = & \Bigl|\int_0^R e^{-2\pi i \om \tau} f_a^{(\ell)}\circ h_{\tau+t'}(\ov{\ek}',0)\,d\tau \Bigr|\\
                                                     & = & \Bigl|\int_{t'}^{R+t'} e^{-2\pi i \om \tau} f_a^{(\ell)}\circ h_{\tau}(\ov{\ek}',0)\,d\tau \Bigr|.
\end{eqnarray*}
Recall that $\vec{s}^{(\ell)} = (\Sf^{[\ell]})^t \vec{s}$, and we let
$s^{(\ell)}_{\max}$ and  $s^{(\ell)}_{\max}$ be the maximal and minimal components of the vector $\vec{s}^{(\ell)}$, respectively.
Note that $|t'|\le s_{\max}^{(\ell)}$, so we obtain
\be \label{kau1}
\Bigl| S^{(\ov{\ek},t)}_R(f_a^{(\ell)},\om) - S^{(\ov{\ek}',0)}_R(f_a^{(\ell)},\om)\Bigr| \le 2\|f^{(\ell)}\|_{_\infty} s_{\max}^{(\ell)}.
\ee
Next, consider $x^{(\ell)} \in \Yk^{(\ell)}$ as in (\ref{SR11}) and take the maximal $N$ such that $R':= |x^{(\ell)}[0,N-1]|_{\vec{s}^{(\ell)}} \le R$. Then
$|R-R'| \le  s_{\max}^{(\ell)}$, hence
\be \label{kau2}
\Bigl| S^{(\ov{\ek}',0)}_R(f_a^{(\ell)},\om) - S^{(\ov{\ek}',0)}_{R'}(f_a^{(\ell)},\om)\Bigr| \le \|f^{(\ell)}\|_{_\infty} s_{\max}^{(\ell)},
\ee
and for $S^{(\ov{\ek}',0)}_{R'}(f_a^{(\ell)},\om)$ the formula in (\ref{SR11}) applies (with $R$ replaced by $R'$).
Thus, the combined error in the above estimates (\ref{kau1}), (\ref{kau2}) is bounded by $3\|f^{(\ell)}\|_\infty \cdot s_{\max}^{(\ell)}$.
By Oseledets Theorem, we can make sure that $\ell_\eta$ is such that 
\be \label{Sell}
\left|\ell^{-1} \log\|\Sf^{[\ell]}\|_1 - \theta_1\right| \le \theta_1\eta/10,\ \ \ \mbox{for all}\ \ \ell\ge \ell_\eta.
\ee
Then
$$
s_{\max}^{(\ell)} \le \|\Sf^{[\ell]}\|_1 \le e^{\theta_1\ell(1+\eta/10)} \le e^{2\theta_1 \ell} < R^{1/2},
$$
for  $\ell\ge \ell_\eta$ and $R\ge e^{6\theta_1 \ell}$. Taking $O_Q(1)\ge 3$, we thus guarantee that the first term in the right-hand side of (\ref{eq-Dioph3}), equal to 
$O_Q(1)\cdot \|f^{(\ell)}\|_{_\infty} R^{1/2}$, dominates the combined error.
Thus
it suffices to   consider the
case of (\ref{SR11}).  Since $R'\le R$ and $|R-R'| \le s_{\max}^{(\ell)} < R^{1/2}$, and $(R-R^{1/2}) \ge R^{3/4}$ for $R\ge 9$, the proposition will follow from the
following lemma. 
\end{proof}

\begin{lemma} \label{lem-vspom5} Suppose that the assumptions of Proposition~\ref{prop-Dioph4} are satisfied. Then
for $\P$-a.e.\ $\bom\in \Om$, 
 for any $\eta\in (0,1)$, there exists $\ell_\eta=\ell_\eta(\bom)\in \N$, such that 
 for all $\ell\ge \ell_\eta$ and any bounded cylindrical function $f^{(\ell)}$ of level $\ell$, 
 for any 
$\ov{\ek}'\in X_+(\bom)$ such that $\hf(\ov\ek') = \zeta^{[\ell]}(x^{[\ell]})$, with $x^{(\ell)}\in \Yk^{(\ell)}$ and $\om\in \R$,

\be \label{eq-Dioph33}
|S^{(\ov{\ek}',0)}_R(f^{(\ell)},\om)| \le  O_Q(1)\cdot \|f^{(\ell)}\|_{_\infty}  R^{1+\eta}\!\!\!\!\!\!\prod_{\ell+1\le k < \frac{\log R}{3\theta_1}} 
\Bigl( 1 - c_1  \cdot \max_{v\in GR(\zeta)}\bigl\| \om|\zeta^{[k]}(v)|_{\vec{s}}\bigr\|^2_{\R/\Z}\Bigr),
\ee
whenever
$$
R=\bigl|x^{(\ell)}[0,N-1]\bigr|_{\vec{s}^{(\ell)}} \ge e^{6\theta_1 \ell}.
$$
\end{lemma}

\begin{proof} Again we  assume without loss of generality that $f^{(\ell)} = f^{(\ell)}_a$, where  
$f_a^{(\ell)}(\ov{\ek},t) = \One_{\Xx^{(\ell)}_a} \psi^{(\ell)}_a(t')$, with $t'\in [0,s^{(\ell)}_a]$ and $\psi^{(\ell)}_a\in C([0, s^{(\ell)}])$, see Section 3.5 for details. 
Recall the formula (\ref{SR11}) which applies here:
\be \label{SR111}
S_R^{(\ov{\ek}',0)}(f^{(\ell)}_a,\om) = \widehat{\psi}^{(\ell)}_a(\om) \cdot {\Phi_a^{\vec{s}^{(\ell)}}\bigl(x^{(\ell)}[0,N-1],\om\bigr)}.
\ee
 First observe that
\be \label{kau3}
|\widehat{\psi}^{(\ell)}_a(\om)|\le \|\psi^{(\ell)}_a\|_1 \le \|\psi^{(\ell)}_a\|_{_\infty} s^{(\ell)}_a \le \|f_a^{(\ell)}\|_{_\infty} s_{\max}^{(\ell)}.
\ee
Next we apply (\ref{eq-new192}), which we copy here for convenience:
\be \label{1922}
\bigl|{\Phi_a^{\vec{s}^{(\ell)}}\bigl(x^{(\ell)}[0,N-1],\om\bigr)}\bigr| \le  2\sum_{j=\ell}^n \|\Sf^{[\ell+1,j]}\|_1\cdot \|\Sf_{j+1}\|_1\ \cdot\!\!\!\!\!\!\!\!
\prod_{\ell+1\le k\le j-1}\!\! \!\!\Bigl( 1 - c_1  \cdot \max_{v\in GR(\zeta)}\bigl\| \om|\zeta^{[k]}(v)|_{\vec{s}}\bigr\|^2_{\R/\Z}\Bigr),
\ee
where
\be \label{Nest}
\min_{b\in \Ak}|\zeta^{[\ell+1,n]}(b)| \le N \le 2\max_{b\in \Ak} |\zeta^{[\ell+1,{n+1}]}(b)|.
\ee
By (\ref{immed1}), we can assume that
\be \label{immed2}
\|\Sf_{j+1}\|_1 \le e^{j\theta_1(\eta /10)}\ \ \ \mbox{for all}\ \ j\ge \ell_\eta.
\ee
Further, note that
\be \label{immed3}
\|\Sf^{[\ell+1,j+1]}\|_1 = \|\Sf^{[\ell+1,j]}Q\Sf_{\xi_{j+1}} Q\|_1 \ge 2\|\Sf^{[\ell+1,j]}\|_1,
\ee
since $Q$ has strictly positive entries, and hence all entries of $Q\Sf_{\xi_{j+1}}Q$ are not less than $m\ge 2$.

From (\ref{immed2}) and (\ref{immed3}), recalling that $c_1 \le 1/4$, we obtain that the sum in (\ref{1922}) is bounded above by $O_Q(1)$ times the last, $n$-th term, yielding
\be \label{1923}
\bigl|{\Phi_a^{\vec{s}^{(\ell)}}\bigl(x^{(\ell)}[0,N-1],\om\bigr)}\bigr| < O_Q(1)\cdot \|\Sf^{[\ell+1,n]}\|_{_1} \cdot e^{n\theta_1(\eta/10)}\cdot\!\!\!\!\!\!\!\prod_{\ell+1\le k\le n-1}\!\! \!\!\Bigl( 1 - c_1  \cdot \!\!\!\max_{v\in GR(\zeta)}\bigl\| \om|\zeta^{[k]}(v)|_{\vec{s}}\bigr\|^2_{\R/\Z}\Bigr).\ee

This, together with (\ref{SR111}), (\ref{kau3}), is already very close to the desired (\ref{eq-Dioph33}), but 
we need to relate $N, n$, and $R$.
First observe that
\be \label{sizeR}
R=\bigl|x^{(\ell)}[0,N-1]\bigr|_{\vec{s}^{(\ell)}}\in [ N s^{(\ell)}_{\min}, N s^{(\ell)}_{\max}].
\ee
Note that
\be \label{ests}
s^{(\ell)}_{\max} \le \col(Q^t) \cdot s^{(\ell)}_{\min},
\ee
since $s^{(\ell)}_{\max}$ and $s^{(\ell)}_{\max}$ are respectively the maximal and minimal components of
$\vec{s}^{(\ell)} = (\Sf^{[\ell]})^t \vec{s} = Q^t(\Sf^{[\ell-1]})^t \vec{s}$  (recall that
$\Sf_\ell = Q$ by assumption).
Further, by (\ref{Nest}) and (\ref{eq-col}),
$$
N \ge (\col(Q^t))^{-1}\cdot \|\Sf^{[\ell+1,n]}\|_1,
$$
since $|\zeta^{[\ell+1,n]}(b)|$ is a column sum of $\Sf^{[\ell+1,n]}$, a matrix which starts and ends with $Q$.
Therefore,
\be \label{there}
R \ge N s_{\min}^{(\ell)} \ge (\col(Q^t))^{-2}\cdot \|\Sf^{[\ell+1,n]}\|_1\cdot s_{\max}^{(\ell)}.
\ee
Comparing (\ref{SR111}), (\ref{kau3}), and (\ref{1923}), we see that
in order to conclude the proof of (\ref{eq-Dioph33}), it remains to show, first,
that
\be\label{remains0}
R^\eta \ge e^{n\theta_1(\eta/10)},
\ee
and second, that 
\be \label{remains}
n\ge (\log R)/(3\theta_1).
\ee
Since $s_{\max}^{(\ell)} \ge (\col(Q^t))^{-1}\|\Sf^{[\ell]}\|_1$, we obtain from (\ref{there}) and (\ref{Sell}) that
$$
R \ge (\col(Q^t))^{-3}\cdot \|\Sf^{[\ell+1,n]}\|_1\cdot |\Sf^{[\ell]}\|_1 \ge (\col(Q^t))^{-3}\cdot \|\Sf^{[n]}\|_1 \ge (\col(Q^t))^{-3}\cdot e^{\theta_1 n (1 - \eta/10)},
$$
confirming (\ref{remains0}) once $\ell$, and hence $n$ is sufficiently large.
On the other hand,
by the upper bound in (\ref{Nest}) and (\ref{sizeR}),
$$
R \le N s_{\max}^{(\ell)} \le 2 \|\Sf^{[\ell]}\|_1\cdot \|\Sf^{[\ell+1,n]}\|_1.
$$
Since all matrices involved begin and end with $Q$, it is easy to see that 
$$
\|\Sf^{[\ell]}\|_1\cdot \|\Sf^{[\ell+1,n]}\|_1\le \col(Q^t)\cdot \|\Sf^{[\ell]}\cdot \Sf^{[\ell+1,n]}\|_1=\col(Q^t)\cdot \|\Sf^{[n]}\|_1,
$$
and then (\ref{Sell}) yields that
$$
R \le O_Q(1)\cdot e^{\theta_1 n (1 + \eta/10)},
$$
which certainly guarantees (\ref{remains}),  once $\ell$, and hence $n$, is sufficiently large.
Now the lemma and the proposition are proved completely.
\end{proof}


\section{Linear algebra and the choice of good return words}

In this section, as well as the next one, we continue to work with a $\P$-generic 2-sided sequence $\bom\in \Om_\bq$.
In view of the assumption (a$'$) of Theorem~\ref{th-main11}, we can fix unit basis vectors $\vec{e}_j^{(n)}$, $j=1,2$, for the one-dimensional subspaces
$E_{\sig^{n}\bom}^j$, $j=1,2$, $n\ge 0$, such that 
\be \label{Osel1}
\AA(n,\bom) \vec{e}_j^{(0)}=A(n,j) {\vec{e}_j}^{(n)}\ \ \mbox{for some}\ \ A(n,j)>0.
\ee
By (ii) in Oseledets Theorem, we have $\frac{1}{n}\log A(n,j) \to\theta_j$, $j=1,2$.

We start with an observation about the Lyapunov-Oseledets basis $\{\vec{e}_j^{(n)}\}_{j=1}^2$ of the unstable
subspace. All the matrices $\AA(n,\bom)$ are non-negative. Thus,
$$
\|\AA(n,\bom)\vec{x}\| \le \|\AA(n,\bom)\vec{x}^{|\cdot|}\|,
$$
hence $\vec{e}_1^{(n)}\in \R^m_+$ (the positive cone) for all $n\ge 0$, and so
\be \label{cone1}
\vec{e}_1^{(n)}\in A(\bom_n) \R^m_+ = Q^t A(\bp_n) Q^t \R^m_+ \subset Q^t \R^m_+.
\ee
On the other hand, under our assumptions the image of the positive cone 
$\AA(n,\bom)\R^m_+$ shrinks to a single direction exponentially fast. (The fact that the cone shrinks to a single direction is  equivalent to unique ergodicity, see (\ref{equiv2}) and Veech \cite {veech,
veechamj}.)
It follows that the basis vector $\vec{e}_2^{(n)}$ does not lie in $\R^m_+$ for all $n$, otherwise we would get a contradiction with (iii) in
Oseledets Theorem. Combined with (\ref{cone1}), this implies that 
the angle between $\vec{e}_1^{(n)}$ and $\vec{e}_2^{(n)}$ is bounded away from zero by a constant depending only on $Q$.

\medskip

We will  next need an elementary fact from linear algebra.

\begin{lemma} \label{lem-lina1}
Let $\Bk=\{\vec{x}_j\}_{j\le m}$ be a basis of $\R^m$, and let $\{\vec\xi_1,\ldots,\vec\xi_r\}\subset \R^m$ be a linearly independent set, with
$r\le m$. Then there exists a subset $\{x_i\}_{i\in I}\subset \Bk$ of cardinality $r$ such that
$$
|D_I|:=\left|\det\left(\langle \vec{x}_i, \vec{\xi}_j\rangle \right)_{i\in I, j\le r}\right| \ge C_\Bk \|\vec{\xi}_1\wedge\cdots \wedge \vec{\xi}_r\|,
$$
where $C_\Bk$ depends only on the basis $\Bk$.
\end{lemma}

\begin{proof}
Let $\Tk$ be the linear isomorphism which takes the standard basis $\{\bbbe_j\}_{j\le m}$ of $\R^m$ into $\Bk$. Then
$$
D_I=\det\left(\langle \vec{x}_i, \vec{\xi}_j\rangle \right)_{i\in I, j\le r} =
\det\left(\langle \Tk \bbbe_i, \vec{\xi}_j \rangle \right)_{i\in I, j\le r}
= \det\left(\langle \bbbe_i, \Tk^* \vec{\xi}_j\rangle \right)_{i\in I, j\le r}
$$
The latter determinant is the order-$r$ minor of the matrix whose columns are $\Tk^* \vec{\xi}_j$, $j=1,\ldots,r$, corresponding to the
rows indexed by $I$. Thus,
$$
\sum_{\#I=r} |D_I|^2 = \|\Tk^*\vec{\xi}_1 \wedge \cdots \wedge \Tk^*\vec{\xi}_r \|^2 \ge  \bigl\|({\textstyle{\bigwedge}}^r\Tk^*)^{-1}\bigr\|^{-1}\|\vec{\xi}_1\wedge\cdots \wedge \vec{\xi}_r\|^2.
$$
We are using here that $\Tk^*$ is invertible, hence its exterior power is invertible. Thus,
$$
\max\{|D_I|:\ \#I=r\} \ge \bigl\|({\textstyle{\bigwedge}}^r\Tk^*)^{-1}\bigr\|^{-1/2} {m\choose r}^{-1/2} \|\vec{\xi}_1\wedge\cdots \wedge \vec{\xi}_r\|,
$$
and the proof is complete.
\end{proof}

We now return to our theorem, in which $r=2$.
Let $\{u_j\}_{j\le m}$ be the good return words from the Assumption (c$'$) of Theorem~\ref{th-main11}.
We will choose  
a sequence of  words $v_n\in \{u_j\}_{j\le m}$, depending on our generic $\bom\in \Om$. 
For $n\ge 1$ consider
\be \label{def-Thb}
\Thb_n := \left( \begin{array}{cc} A(n,1) \langle \vec{\ell}(v_n), \vec{e}_1^{(n)} \rangle &
                                          A(n,2) \langle \vec{\ell}(v_n), \vec{e}_2^{(n)} \rangle \\
                                          A({n+1},1) \langle \vec{\ell}(v_{n+1}), \vec{e}_1^{({n+1})} \rangle &
                                          A({n+1},2) \langle \vec{\ell}(v_{n+1}), \vec{e}_2^{({n+1})} \rangle \end{array} \right).
\ee
Below $C_\zeta$ denotes a constant depending only on the substitution $\zeta=\zeta(\bq)$.

\begin{lemma} \label{lem-new1}
For $\P$-a.e.\ $\bom\in \Om$ we can choose the  words $v_n\in \{u_j\}_{j\le m}$, so that for all $n\ge 1$,
\be \label{nov1}
\|\Thb_n^{-1}\|_\infty \le C_\zeta\cdot \frac{\max \{A({n+j},i);\ {j=0,1; i=1,2}\}}{A(n,1) A({n+1},2)}
\ee
and
\be \label{nov2}
\|\Thb_{n+1} \Thb_n^{-1}\|_\infty  \le C_\zeta\cdot \frac{\max_{j=0,1,2}A({n+j},1)}{A(n,1)}\cdot
                                                  \frac{\max_{j=0,1,2}A({n+j},2)}{A(n,2)}\,.
\ee
\end{lemma}

\begin{proof} We are going to choose $v_n$ inductively.
Pick $v_1$ arbitrarily, and suppose $v_1,\ldots,v_n$
have been chosen.  For $i\le m$ consider
$$
\Delta_i := \det \left( \begin{array}{cc} A(n,1) \langle \vec{\ell}(v_n), \vec{e}_1^{(n)} \rangle &
                                          A(n,2) \langle \vec{\ell}(v_n), \vec{e}_2^{(n)} \rangle \\
                                          A({n+1},1) \langle \vec{\ell}(u_i), \vec{e}_1^{({n+1})} \rangle &
                                          A({n+1},2) \langle \vec{\ell}(u_i), \vec{e}_2^{({n+1})} \rangle \end{array} \right).
$$
Observe that
\be\label{urka1}
\det \left( \begin{array}{cc} \langle \vec{\ell}(u_i), \vec{e}_1^{({n+1})} \rangle & \Delta_i \\
                                           \langle \vec{\ell}(u_j), \vec{e}_1^{({n+1})} \rangle & \Delta_j\end{array} \right) = 
      A(n,1) A({n+1},2) \langle \vec{\ell}(v_n), \vec{e}_1^{(n)} \rangle D_{ij},
\ee
where
$$
D_{ij}:= \det\left( \begin{array}{cc} 
              \langle \vec{\ell}(u_i), \vec{e}_1^{({n+1})} \rangle & \langle \vec{\ell}(u_i), \vec{e}_2^{({n+1})} \rangle \\
              \langle \vec{\ell}(u_j), \vec{e}_1^{({n+1})} \rangle & \langle \vec{\ell}(u_j), \vec{e}_2^{({n+1})} \rangle 
     \end{array} \right).
$$              
Note that $\vec{\xi}_1:=\vec{e}_1^{({n+1})} \in Q^t \R^m_+$ and $\vec{\xi}_2:=\vec{e}_2^{({n+1})}\not \in \R^m_+$ by the  comments above. Thus, the angle between $\vec{\xi}_1$ and $\vec{\xi}_2$ is bounded away from zero,
uniformly in $n$. Hence we
can apply Lemma~\ref{lem-lina1} to these vectors and find $i\ne j$ such that $$|D_{ij}|\ge c_3>0,$$ independent of $n$.
Note that for all $i\le m$,
\be \label{urka2}
0 < c_4 \le | \langle \vec{\ell}(u_i), \vec{e}_1^{(n)} \rangle |\le C_5:=\max_{i\le m}\|\vec{\ell}(u_i)\|_2
< \infty
\ee
for some positive constant $c_4=c_4(Q)$ independent of $n$, since $\vec{\ell}(u_i),\ i\le m,$ are positive integer vectors.
It follows from (\ref{urka1}) and (\ref{urka2}) that
$$
\max_i |\Delta_i| \ge \frac{c_3 c_4}{2C_5} A(n,1) A({n+1},2).
$$
We choose $v_{n+1}\in \{u_i\}_{i\le m}$ to maximize $|\Delta_i|$. Denote $\Delta^{(n)} = \det(\Thb_n)$. As a result, we will have for all $n\ge 1$:
\be \label{ineq1}
|\Delta^{(n)}| \ge \frac{c_3 c_4}{2C_5} A(n,1) A({n+1},2),
\ee
which implies (\ref{nov1}). Also, a direct calculation, combined with (\ref{urka2}) and (\ref{ineq1}), yields (\ref{nov2}).
\end{proof}

\begin{corollary} \label{cor-linal}
For $\P$-a.e.\ $\bom\in \Om_\bq$ we can choose the  words $v_n\in \{u_j\}_{j\le m}$, so that for any $\delta_1>0$ there exists $n_0\in \N$
such that for all $n\ge n_0$,
\be \label{nov3}
\|\Thb_n^{-1}\|_\infty \le C_\zeta \exp[-(\theta_2-\delta_1)n]
\ee
and
\be \label{nov4}
\|\Thb_{n+1} \Thb_n^{-1}\|_\infty  \le C_\zeta\exp[2(W_n+W_{n+1})],
\ee
where $C_\zeta$ is the constant from Lemma~\ref{lem-new1} and $W_n$ are defined in (\ref{def-W}).
\end{corollary}

\begin{proof}
This is a combination of the last lemma and  Oseledets Theorem. First we prove (\ref{nov3}).
By Oseledets Theorem, for $\P$-a.e.\ $\bom\in \Om_\bq$, for all $n$ sufficiently large,
$$
\exp[(\theta_i-\delta_1/4)n] \le A(n,i)\le \exp[(\theta_i+\delta_1/4)n],\ \ i=1,2.
$$
We will use (\ref{nov1}), where clearly the maximum in the numerator is (eventually) attained
for $i=1$, to obtain for $n$ sufficiently large:
\begin{eqnarray*}
\| \Thb_n^{-1}\|_\infty  & \le & C_\zeta \exp[(\theta_1 + \delta_1/4) {n+1} - (\theta_1 - \delta_1/4) n- (\theta_2 - \delta_1/4) {n+1}]\\
& = & C_\zeta \exp[(\theta_1 -\theta_2 + \delta_1/2) - (\theta_2-3\delta_1/4)n].
\end{eqnarray*}
For $n$ sufficiently large,
$$
\theta_1 -\theta_2 + \delta_1/2\le (\delta_1/4)n,
$$
and (\ref{nov3}) follows.
Next, let us verify (\ref{nov4}). Equation (\ref{Osel1}) implies that
$$
A({n+1},j)\vec{e}_j^{({n+1})}= A(\bom_n) A({n},j)\vec{e}_j^{(n)}, \ \ j=1,2,
$$
hence $A({n+1},j)/A(n,j) \le e^{W_n}$ by (\ref{def-W}). Now (\ref{nov4}) follows from (\ref{nov2}).
\end{proof}


\section{Beginning of the proof of Theorem~\ref{th-main11}}

Now we proceed with the proof of Theorem~\ref{th-main11}. As before, we fix a $\P$-generic point $\bom\in \Om_\bq$. 
Under the assumptions of Theorem~\ref{th-main11},
\be \label{expa1}
\vec{s} = \sum_{j=1}^2 a_j \vec{e}_j^{(0)} + P_\bom^{st}\vec{s},
\ee
where $a_j = P^j_\bom(\vec{s})$ and $P_{\bom}^{st}$ is the projection to the stable subspace $E^3_{\bom} \oplus \cdots \oplus E^r_\bom$ in (\ref{os1}).  Recall that $\vec{s}\in \Delta_m:= \{\vec{s}\in \R^m_+:\ \|\vec{s}\|_1=1\}$.
Our goal is to prove that for all $\beta>0$ and $B>1$, for a.e.\ $\vec{s}$, with $\min_{j=1,2}|a_j|\ge \beta$,
\be \label{Hoeld1}
\sig_f(B(\om,r)) \le C(\bom,\|f\|_L)\cdot r^\gam\ \ \mbox{for}\ \om\in [B^{-1},B]\ \mbox{and}
 \ 0 < r \le r_0(\bom,\beta,B).
\ee
 Fix $\beta>0$ and $B>1$ for the rest of the proof. Recall that dependence on $\vec{s}$ in the estimate is ``hidden'' in $\sig_f$, which is the spectral measure of the suspension flow corresponding to the roof function given by $\vec{s}$.

Fix the sequence of good return words $v_n$ from Lemma~\ref{lem-new1}.
For $n\in \N$ and $\om\in [B^{-1},B]$, let
\be \label{dio1}
\om |\zeta^{[n]}(v_n)|_{\vec{s}} = K_{n} + \eps_{n},\ \ \mbox{where}\ K_{n}\in \N,\ |\eps_{n}|\le 1/2,
\ee
so that $${\|\om |\zeta^{[n]}(v_n)|_{\vec{s}} \|}_{\R/\Z}=|\eps_{n}|.$$ We should keep in mind that $K_{n}$ and $\eps_{n}$ depend on
$\om$ and on $\vec{s}$, although this is suppressed in notation to avoid clutter.
Given  $\beta,\varrho,\delta>0$ and $B>1$, define
\begin{eqnarray*}
E_N(\varrho,\delta,\beta,B) & := & \Bigl\{\vec{s}\in \Delta_{m}:\ \min_{j=1,2}|P_\bom^j(\vec{s})|\ge \beta \ \mbox{and}\ \exists\,\om\in [B^{-1},B]\\ & & \mbox{such that}\ \ 
\card\{n\le N: |\eps_{n}|\ge \varrho\} < \delta N\Bigr\}.
\end{eqnarray*}
and 
$$
\Ek(\varrho,\delta,\beta,B):= \bigcap_{N_0=1}^\infty  \bigcup_{N=N_0}^\infty E_N(\varrho,\delta,\beta,B).
$$

\begin{prop} \label{prop-EK}
There exist $\varrho>0$  such that for $\P$-a.e.\ $\bom\in \Om_\bq$  we have
$$
\forall\,\epsilon>0,\ \exists\,\delta_0>0,\ \forall\,\beta>0,\ \forall\,B>1:\ \ \ \ \delta<\delta_0\ \ \Longrightarrow\ \ 
\dim_H(\Ek(\varrho,\delta,\beta,B))< m-2+\epsilon.
$$
\end{prop}

In the remaining part of this section, we derive Theorem~\ref{th-main11} from Proposition~\ref{prop-EK}. Then in the next section we use
the ``Erd\H{o}s-Kahane argument'' to prove the proposition.

\begin{proof}[Proof of Theorem~\ref{th-main11} assuming Proposition~\ref{prop-EK}]
In view of Lemma~\ref{lem-varr}, it suffices to show
\be \label{wts1}
|S^{(\ov{\ek},t)}_R(f,\om)| \le \wtil{C}(\bom,\|f\|_L)\cdot R^{1-\gam/2},\ \ \mbox{for}\ R\ge R_0(\bom,\beta,B),
\ee
for some $\gam\in (0,1)$, uniformly in $(\ov{\ek},t)\in \Xx^{\vec{s}}$. We will specify $\gam$ at the end of the proof, see (\ref{def-eta}).

Since $f$ is weakly Lipschitz on $\Xx^{\vec{s}}$ (see Section 2.1), for almost every $\bom$
 we can approximate $f$, for any $\ell\in \N$, by a function $f^{(\ell)}$, which is cylindrical of level $\ell$, and has sup-norm bounded by $\|f\|_\infty$, so that
$$
\|f-f^{(\ell)}\|_\infty \le \|f\|_L\cdot \nu_+([\ek_1\ldots \ek_n]).
$$
We can do this simply taking $f^{(\ell)}(\ov{\ek},t):= f(\ov{\ek}^{(\ell)},t)$, where $\ov{\ek}^{(\ell)}$ agrees
with $\ov{\ek}$ down to level $\ell$ after which it is extended to infinity in any fixed way.  By (\ref{LL2}), (\ref{meas1}),  and (\ref{equiv2}), we have 
$$
\lim _{n\to \infty} \frac{\log \nu_+([\ek_1\ldots \ek_n])}{n} = -\theta_1,
$$
$\P$-almost surely, hence for $\ell$ sufficiently large we have
\be \label{approx}
\|f-f^{(\ell)}\|_\infty \le \|f\|_L\cdot e^{-\half\theta_1\ell}.
\ee
Recall that $S_R^{(\ov{\ek},t)}(f,\om) = \int_0^R e^{-2\pi i \om t} f\circ h_\tau(\ov{\ek},t)\,d\tau$. 
Let
\be\label{def-ell}
\ell := \left\lfloor \frac{2\gam \log R}{\theta_1}\right\rfloor.
\ee
Then (\ref{approx}) yields
$$
|S_R^{(\ov{\ek},t)}(f,\om) - S_R^{(\ov{\ek},t)}(f^{(\ell)},\om)| \le R\cdot \|f\|_L \cdot e^{-\half\theta_1 \ell} \le e^{\theta_1/2} \cdot \|f\|_L\cdot R^{1-\gam}.
$$
Thus, it is enough to obtain (\ref{wts1}),  with $f$ replaced by $f^{(\ell)}$.
For the latter, we can apply Proposition~\ref{prop-Dioph4}. Recall the inequality (\ref{eq-Dioph3}), using the sequence of good return words $\{v_n\}$ and $\eta = \gam/2$:
$$
|S^{(\ov{\ek},t)}_R(f^{(\ell)},\om)| \le  O_Q(1)\cdot \|f^{(\ell)}\|_{_\infty} \Bigl(R^{1/2}+ R^{1+\gam/2}\!\!\!\!\!\!\prod_{\ell+1\le n < \frac{\log R}{4\theta_1}} 
\bigl( 1 - c_1 \cdot{\|\om|\zeta^{[n]}(v_n)|_{\vec{s}}\|}_{\R/\Z}^2\bigr)\Bigr),
$$
for $\ell\ge \ell_\gam$ and all
$
R\ge e^{8\theta_1 \ell}.
$
We can ensure that $\ell\ge \ell_\gam$ by taking $R_0$ sufficiently large, and $R\ge e^{8\theta_1 \ell}$ will follow from (\ref{def-ell}) if $\gam\le 1/16$.
 Since our goal is (\ref{wts1}),
we can discard the $R^{1/2}$ term immediately.
Now choose $\varrho>0$ and $\delta>0$ from Proposition~\ref{prop-EK} such that $\dim_H(\Ek(\varrho,\delta,\beta,B))<m-1=\dim(\Delta_{m})$
for all $\beta>0, B>1$.
It is enough to verify 
\be \label{wts2}
\prod_{\ell+1\le n < \frac{\log R}{4\theta_1}} 
\bigl( 1 - c_1 \cdot{\|\om|\zeta^{[n]}(v_n)|_{\vec{s}}\|}_{\R/\Z}^2\bigr)\le O_{\bom,\|f\|_L}(1)\cdot R^{-\gam},\ \ \om\in [B^{-1},B],
\ee
for
$ R \ge R_0(\bom,\beta,B),
$
 for all vectors $\vec{s}\in \Delta_{m}\setminus\Ek(\varrho,\delta,\beta,B)$, for which $\min_{j=1,2}|P_\bom^j (\vec{s})|\ge \beta$, 
thus obtaining an even stronger than `almost every $\vec{s}$\,' statement. 

By definition, $\vec{s} \not\in \Ek(\varrho,\delta,\beta,B)$ means that there exists $N_0=N_0(\bom,\beta,B)\in \N$ such that  $\vec{s}\not\in E_N(\varrho,\delta,\beta,B)$ for all $N\ge N_0$.  
Let
\be \label{defiN}
N = \left\lfloor\frac{\log R}{4\theta_1}\right\rfloor \ \ \ \mbox{and}\ \ \ R_0=e^{4\theta_1(N_0+1)}.
\ee
Then $R\ge R_0$ implies $N\ge N_0$, and the product in (\ref{wts2}) is less than or equal to
$$
\prod_{n=\ell+1}^{N-1} (1- c_1 |\eps_{n}|^2),
$$
where we also use (\ref{dio1}).
 By definition,
$\vec{s}\not\in E_N(\varrho,\delta,\beta,B)$ means that there are at least $\lceil\delta N\rceil$ numbers $n\in \{1,\ldots,N\}$ with
$
|\eps_{k_n}| \ge \varrho,
$
hence the left-hand side of (\ref{wts2}) is bounded above by $(1-c_1 \varrho^2)^{\delta N-1-\ell}$. Recalling (\ref{defiN}) and (\ref{def-ell}),  we see that 
\be \label{uff}
(1-c_1 \varrho^2)^{\delta N-1-\ell} \le O(1)\cdot (1-c_1 \varrho^2)^{(\delta \log R)/(8\theta_1)} \le O(1)\cdot R^{-\gam},
\ee
provided that 
\be \label{def-eta}
\gam \le \min\Bigl\{\frac{\delta}{16}, \frac{-\delta\log(1-c_1\varrho^2)}{8\theta_1}\Bigr\},
\ee
with the two conditions for $\gam$ needed for the left and right inequality in (\ref{uff}) correspondingly.
Now the proof of (\ref{wts2}) is complete, and it remains to verify Proposition~\ref{prop-EK} to conclude the proof of Theorem~\ref{th-main11}.
\end{proof}


\section{The Erd\H{o}s-Kahane method and the conclusion of the proof of Theorem~\ref{th-main11}}

In this section, we prove Proposition~\ref{prop-EK}.
We need some preparation first.
Fix $\beta\in (0,1)$ and suppose that $\vec{s}\in \Delta_m$ is such that $\min_{j=1,2}|P_\bom^j(\vec{s})|\ge \beta$. Further, let $B>1$ and $\om \in [B^{-1},B]$. Recall (\ref{expa1}) and (\ref{dio1}).
 In view of (\ref{Osel1}),
we have for $n\ge 1$,  denoting $\xi_{n} = \langle \vec{\ell}(v_n),  (\Sf^{[n]})^t P_\bom^{st}\vec{s}\rangle$:
\begin{eqnarray}
|\zeta^{[n]}(v_n)|_{\vec{s}} & = & \langle \vec{\ell}(\zeta^{[n]}(v_n)),\vec{s} \rangle  = \langle S^{[n]} \vec{\ell}(v_n), \vec{s}\rangle =
\langle \vec{\ell}(v_n), (\Sf^{[n]})^t \vec{e}_j^{(0)}\rangle \nonumber \\
& = & \sum_{j=1}^2 a_j A(n,j) \langle\vec{\ell}(v_n), {\vec{e}_j}^{(n)}\rangle+ \xi_{n}. \label{expand}
\end{eqnarray}
By the Assumption (a$'$) of Theorem~\ref{th-main11}, for $\P$-a.e.\ $\bom\in \Om_\bq$, we have
$\limsup_{n\to \infty} n^{-1} \log\|(\Sf^{[n]})^tP^{st}_\bom \|\le \theta_3<0$,  hence
\be \label{decay1}
\limsup_{n\to \infty} n^{-1} \log|\xi_{n}|\le \theta_3<0.
\ee
Now let $\vec{s}\in \Delta_m$ and $\om\in [B^{-1},B]$ be from the definition of $E_N(\varrho,\delta,\beta,B)$. Recall (\ref{dio1}), and denote
$$
\vec{a} = \left( \begin{array}{c} a_1\\  a_2\end{array} \right),\ \ \  
\vec{K}_n = \left( \begin{array}{c} K_{n}\\  K_{{n+1}}\end{array} \right),\ \ \ 
\vec{\eps}_n = \left( \begin{array}{c} \eps_{n}\\ \eps_{{n+1}}\end{array} \right),\ \ \ 
\vec{\xi}_n = \left( \begin{array}{c} \xi_{n}\\  \xi_{{n+1}}\end{array} \right).
$$
We need the matrices $\Thb_n$ defined in (\ref{def-Thb}):
$$
\Thb_n = \left( \begin{array}{cc} A(n,1) \langle \vec{\ell}(v_n), \vec{e}_1^{(n)} \rangle &
                                   A(n,2) \langle \vec{\ell}(v_n), \vec{e}_2^{(n)} \rangle \\
                          A({n+1},1) \langle \vec{\ell}(v_{n+1}), \vec{e}_1^{({n+1})} \rangle &
                          A({n+1},2) \langle \vec{\ell}(v_{n+1}), \vec{e}_2^{({n+1})} \rangle \end{array} \right).
$$
The equations (\ref{dio1}) for $n, n+1$, in view of (\ref{expand}), combine into
$$\om \Thb_n \vec{a} = \vec{K}_n + \vec{\eps}_n -\om\vec{\xi}_n,$$
hence
\be \label{ur2}
\vec{a} = \om^{-1} \Thb_n^{-1}(\vec{K}_n + \vec{\eps}_n -\om\vec{\xi}_n).
\ee
It follows that
\be \label{ur3}
a_j =
\om^{-1}[\Thb_n^{-1}(\vec{K}_n + \vec{\eps}_n -\om\vec{\xi}_n)]_j,\ \ j=1,2,
\ee
where $[\cdot]_j$ denotes $j$-th component of a vector. Observe that
\be \label{abound}
0 < \beta\le |a_1|, |a_2|\le C_\bom,
\ee
where the upper bound comes from the fact that $\|\vec{s}\|_1=1$ and the angles between Lyapunov subspaces at $\bom$ depend on $\bom$. Choose $\delta_1>0$ such that $\theta_2-\delta_1>0$ and $\theta_3+\delta_1<0$. Note that $\|\vec{\eps}_n\|_\infty\le \half$ for all $n$, and for $n\ge n_0(\bom)$,
\be \label{xibound}
|\xi_{k_n}|\le e^{(\theta_3+\delta_1)k_n} \le e^{(\theta_3+\delta_1)n}.
\ee
Since $|\om|\le B$, we have
$$
\|\om \vec{\xi}_n\|_\infty \le 1/2\ \ \mbox{for}\ n\ge n_0(\bom) + C_{\rm Lyap}\cdot\log B,\ \ \mbox{where}\ \ C_{\rm Lyap} = |\theta_3+\delta_1|^{-1}.
$$
Here and below we denote by $n_0(\bom)$ a generic integer constant depending on $\bom$, and possibly also on the Lyapunov spectrum (below it will also depend on $\beta$), and by $C_\Lyap$ a constant which depends only on the Lyapunov spectrum. Similarly, $C_\bom\ge 1$ is a generic constant depending on $\bom$, and possibly also on the Lyapunov spectrum. These constants may be different from line to line.

Thus, $\|\vec{\eps}_n - \om\vec{\xi}_n\|_\infty \le 1$ for large enough $n$, so
(\ref{ur3}), the lower bound for $\om$, and  (\ref{nov3}) yield for  $j=1,2$:
\begin{eqnarray}
|a_j - \om^{-1}(\Thb_n^{-1}\vec{K}_n)_j| & \le & B\|\Thb_n^{-1}\|_\infty \nonumber\\
& \le  & B C_\zeta e^{-(\theta_2-\delta_1)k_n} \label{ajest} \le B C_\zeta e^{-(\theta_2-\delta_1)n},\ \ \mbox{for}\ n\ge n_0(\bom). 
\end{eqnarray}
In view of (\ref{abound}) and (\ref{ajest}), we obtain  for $j=1,2$:
$$
0< \beta/2 \le |\om^{-1}(\Thb_n^{-1}\vec{K}_n)_j|\le 2 C_\bom, \ \ \mbox{for}\ n\ge n_0(\bom,\beta) + C_{\rm Lyap}\cdot\log B.
$$
From these bounds and (\ref{ajest}), we obtain, again for $n\ge n_0(\bom,\beta) + C_{\rm Lyap}\cdot\log B$:
\begin{eqnarray}
\left|\frac{a_2}{a_1} - \frac{[\Thb_n^{-1}\vec{K}_n]_2}{[\Thb_n^{-1}\vec{K}_n]_1}\right| & \le &
 \frac{|a_2 - \om^{-1}(\Thb_n^{-1}\vec{K}_n)_2|}{|a_1|}+ \frac{|\om^{-1}(\Thb_n^{-1}\vec{K}_n)_2|\cdot
|a_1 - \om^{-1}(\Thb_n^{-1}\vec{K}_n)_1|}{|a_1|\cdot| \om^{-1}(\Thb_n^{-1}\vec{K}_n)_1|} \nonumber\\[1.3ex] & \le &
6C_1(\bom)\beta^{-2}B C_\zeta \exp[-(\theta_2-\delta_1)n]. \label{ur31}
\end{eqnarray}

On the other hand, comparing (\ref{ur2}) for $n$ and $n+1$ yields
$$
\vec{K}_{n+1} + \vec{\eps}_{n+1} -\om\vec{\xi}_{n+1} = \Thb_{n+1} \Thb_n^{-1}[\vec{K}_n + \vec{\eps}_n -\om\vec{\xi}_n],
$$
hence, using $|\om|\le B$, we obtain
$$
\|\vec{K}_{n+1} - \Thb_{n+1} \Thb_n^{-1}\vec{K}_n \|_\infty \le 
\|\vec{\eps}_{n+1}\|_\infty + B\|\vec{\xi}_{n+1}\|_\infty + \|\Thb_{n+1} \Thb_n^{-1}\|_\infty
(\|\vec{\eps}_n \|_\infty + B \| \vec{\xi}_{n}\|_\infty).
$$
This implies, in view of (\ref{nov4}), for $n\ge n_0(\bom)$:
\begin{eqnarray}
 \left|K_{{n+2}} - [\Thb_{n+1} \Thb_n^{-1}\vec{K}_n]_2\right| 
& \le & (1+C_\zeta \exp[2(W_n+W_{n+1})]\times \nonumber \\
& \times & \left( \max\{|\eps_{n}|,|\eps_{{n+1}}|,|\eps_{{n+2}}| \}+ 
B \max \{|\xi_{n}|,|\xi_{{n+1}}|,|\xi_{{n+2}}| \} \right).\label{ineq5}
\end{eqnarray}
Let
\be \label{def-rho_n}
M_n := 1+C_\zeta \exp[2(W_n+W_{n+1})]\ \ \ \mbox{and}\ \ \ \rho_n = \frac{1}{4M_n}\,.
\ee

\begin{lemma} \label{lem-vspom2}
For all $n\ge n_0(\bom,\beta)+C_{\Lyap}\log B$, we have the following, independent of $\om\in [B^{-1},B]$ and
$\vec{s}\in \Delta_m$, satisfying $|a_j| = |P_\bom^j (\vec{s})|\ge \beta>0$:

{\bf (i)} Given $K_{n}, K_{{n+1}}$, there are at most $2M_n+1$ possibilities for the integer $K_{{n+2}}$;

{\bf (ii)} if $\max\{|\eps_{n}|,|\eps_{{n+1}}|,|\eps_{{n+2}}| \}< \rho_n$, then $K_{{n+2}}$ is uniquely determined by $K_{n}, K_{{n+1}}$.
\end{lemma}

\begin{proof}
For part (i), we just use that $\|\vec{\eps}_n\|_\infty\le \half$ for all $n$ and $\|\vec{\xi}_n\|_\infty \le (2B)^{-1}$ for $n\ge n_0(\bom) +C_{\Lyap}\log B$, and that the number of integer points in an interval of length $2M_n$ is at most $2M_n+1$. 

For part (ii), we claim that
$K_{{n+2}}$ belongs to a neighborhood of radius less than $\half$, centered at $[\Thb_{n+1} \Thb_n^{-1}\vec{K}_n]_2$, under the given assumptions,
for $n$ sufficiently large. We have $M_n\rho_n = 1/4$, so it remains to make sure that 
$$
\max \{|\xi_{n}|,|\xi_{{n+1}}|,|\xi_{{n+2}}| \}\le \rho_n/B
$$
for $n$ sufficiently large. Note that for any $\wtil{\delta}>0$ we have, by (\ref{eqrem}):
$$
\rho_n \ge e^{-\wtil{\delta}n},\ \ \mbox{for}\ n\ge n_0(\bom).
$$
Taking $\wtil{\delta}< |\theta_3+\delta_1|$ and combining the last inequality with (\ref{xibound}) implies the desired claim.
\end{proof}

\begin{proof}[Proof of Proposition~\ref{prop-EK}]
Let $\wt E_N(\delta,\beta,B)$ be defined by 
\begin{eqnarray*}
\wt E_N(\delta,\beta,B) & := & \Bigl\{\vec{s}\in \Delta_{m}:\ \min_{j=1,2}|P_\bom^j(\vec{s})|\ge \beta \ \mbox{and}\ \exists\,\om\in [B^{-1},B]\ \mbox{such that}\\ 
 & &  
 \card\{n\le N:\ \max\{|\eps_{n}|,|\eps_{{n+1}}|,|\eps_{{n+2}}| \}\ge \rho_n\} < 
\delta N\Bigr\}.
\end{eqnarray*}
First we claim that $\P$-almost surely,
\be \label{claima}
\wt E_N(\delta,\beta,B) \supset E_N(\varrho, \delta/6, \beta, B)
\ee
for $N\ge N_0(\bom)$, where
\be \label{defK}
\varrho = (1/4)(1+C_\zeta e^{2K})^{-1}, \ \ \mbox{with}\ \ K =25 L_1 \log(1/\delta).
\ee
Here $C_\zeta$ is from Lemma~\ref{lem-new1} and $L_1$ is from Proposition~\ref{prop-proba}.
Suppose $\vec{s} \not\in \wt E_N(\delta,\beta,B)$. Then for all $\om\in [B^{-1},B]$
there exists a subset $\Gam_N \subset \{1,\ldots,N\}$ of cardinality $\ge \delta N/3$ such that
$$
|\eps_{k_n}| \ge \rho_n\ \ \ \mbox{for all}\ k_n \in \Gam_N.
$$
Observe that 
 there are fewer than $\delta N/6$ integers $n\le N$ for which $W_n+W_{n+1}>K$, for $N\ge N_0(\bom)$.
Indeed, otherwise we can find $\Psi\subset \{1,\ldots,N\}$, with $|\Psi|\ge \delta N/12$, 
such that $W_n >K/2$ for $n\in \Psi$, hence
$$
\sum\{W_n:\ n\in \Psi\} \ge K \delta N/24,
$$
which contradicts (\ref{w-cond2}) for $K> 24L_1 \log(1/\delta)$. In view of (\ref{def-rho_n}) and (\ref{defK}), it follows that
$$
\card\bigl\{n\in \Gam_N:\ \rho_n \ge \varrho\bigr\}\ge \delta N/6.
$$
Thus $\vec{s}\not \in E_N(\varrho, \delta/6, \beta, B)$ which confirms (\ref{claima}).

\medskip

\begin{sloppypar}
It follows that it is enough to estimate the dimension of 
$$
\wt \Ek:=\wt\Ek(\delta,\beta,B):= 
\bigcap_{N_0=1}^\infty   \bigcup_{N=N_0}^\infty \wt E_N(\delta,\beta,B).
$$
 Suppose $\vec{s} \in \wt E_N:=\wt E_N(\delta,\beta,B)$; choose  $\om$ from the definition of $\wt E_N$, and find the 
corresponding sequence $K_{n}, \eps_{n}$. In order to prove that $\dim_H(\wt \Ek)<m-2+\eps$, it is enough to show that the set of $a_2/a_1$ corresponding to $\vec{s}\in \wt \Ek$ has Hausdorff dimension smaller than $\eps$. We have from (\ref{ur31}) that 
$a_2/a_1$ is covered by an interval of radius 
\be \label{eqrad}
C_1(\bom,\beta) \cdot B \exp[-(\theta_2-\delta_1)N], \ \ \mbox{for}\ N\ge N_0(\bom,\beta) + C_\Lyap\log B,
\ee
centered at $[\Thb_n^{-1}\vec{K}_n]_2/[\Thb_n^{-1}\vec{K}_n]_1$.
Thus we need to estimate the number of sequences $K_{n}$, $n\le N$, which may arise. Let $\Psi_N$ be the set of $n\in \{1,\ldots,N\}$ for which
we have $\max\{|\eps_{n}|,|\eps_{{n+1}}|,|\eps_{{n+2}}| \}\ge \rho_n$.
By the definition of $\wt E_N$ we have $|\Psi_N| <\delta N$. 
There are $\sum_{i< \delta N} {N\choose i}$ such sets. For a fixed $\Psi_N$ the number of possible sequences $\{K_{n}\}$ is at most
$$
\Bk_N:= \prod_{n\in \Psi_N} (2M_n+1),
$$
times the number of ``beginnings'' $K_{1},\ldots,K_{{n_2}}$, by Lemma~\ref{lem-vspom2}.
The number of possible ``beginnings'' is  bounded, independent of $N$ by a constant depending on $\beta$ and $B$, in view of the a priori bounds on $\om$ and $\vec{s}$.
By the definition of $M_n$ and (\ref{w-cond2}), we have,  for $N$ sufficiently large:
$$
\Bk_N \lesssim \exp\left(C'' \sum_{n\in \Psi_N} (W_n+W_{n+1})\right)\le \exp\left[\wtil{L}\log(1/\delta)(\delta N)\right].
$$
Thus, by (\ref{eqrad}), the number of balls of radius $O_{\beta,B}(1)e^{-(\theta_2-\delta_1)N}$ needed to cover $\wt \Ek$ is at most
\be \label{lasta1}
O_{\beta,B}(1)\cdot  \sum_{i<\delta N} {N\choose i}  \exp\left[\wtil{L}\log(1/\delta)(\delta N)\right] \le O_{\beta,B}(1)\cdot \exp\left[(\wtil{L}+C') \log(1/\delta)(\delta N)\right],
\ee
using (\ref{Stirling}) in the last inequality. Since $\delta\log(1/\delta)\to 0$ as $\delta\to 0$, we can choose $\delta_0>0$ so small 
that $\delta< \delta_0$ implies
$$
\left[(\wtil{L}+C') \log(1/\delta)(\delta N)\right] < \epsilon (\theta_2-\delta_1)N,
$$
whence  $\wt \Ek$ has Hausdorff dimension less than $\epsilon$, as desired. The proof of Proposition~\ref{prop-EK}, and hence of
Theorem~\ref{th-main11}, is complete.
\end{sloppypar}
\end{proof}


\section{Derivation  of Theorem~\ref{main-moduli}  from Theorem~\ref{th-main1}}
Consider our surface $M$ of genus 2. By the results of \cite[Section 4]{Buf-umn}  there is a correspondence between almost every translation flow and
an element $\bom\in \Om$ (space of 2-sided Markov compacta), such that the (uniquely ergodic) flow is measure-theoretically conjugate to the uniquely ergodic flow $(X(\bom),h_t^+)$ and hence to the suspension flow $(\Xxs,h_t)$ over the Vershik map $(X_+(\bom),\Tf)$ with the roof function corresponding to an appropriate vector $\vec{s}$,
see Lemma~\ref{symb-flow}.
By construction (see \cite{Buf-umn}), this correspondence intertwines the Teichm\"uller flow on the space of Abelian differentials and a measure-preserving system $(\Om,\P,\sigma)$, as considered at the beginning of our Section 4. The key point here is that
the Masur-Veech measure on the space of abelian differentials is taken, under this correpondence,   to a measure mutually absolute continuous with the product of the
measure $\P_+$ on $\Om_+$ and the Lebesgue measure on the $3$-dimensional set of possible vectors $\vec{s}$ defining the suspension.

More precisely: as is well-known, the translation flow on the surface can be realized as a suspension flow over
an interval exchange transformation (IET), see \cite{viana2} for details.
Veech \cite{veech}
constructed, for any connected component of a stratum $\Hk$, a measurable finite-to-one map from the space $\Vk(\Rk)$ of
zippered rectangles corresponding to the Rauzy class $\Rk$, to $\Hk$, which intertwines the Teichm\"uller flow on $\Hk$ and a renormalization flow $P_t$ that Veech defined on $\Vk(\Rk)$. 
Observe that in our case the
stratum $\Hk(2)$ is connected and corresponds to the Rauzy class of the IET with permutation $(4,3,2,1)$. 
In general, the Veech mapping from  $\Vk(\Rk)$ to $\Hk$ is not bijective (it corresponds to passing from absolute
to relative real cohomologies in the manifold $M$), but in our case the kernel is trivial, since the manifold has only
one singularity and therefore there are no saddle connections. For background and complete details the reader is referred to 
\cite{Buf-umn} and \cite{viana2}. 

Section 4.3 of \cite{Buf-umn}  constructs the symbolic coding of the flow $P_t$ on $\Vk(\Rk)$, namely, a map
\be \label{map-veech}
\Xi_{\Rk}: (\Vk(\Rk),\wt\mu_2) \to ({\Omega},\P),
\ee
defined almost everywhere, where $\wt\mu_2$ is the pull-back of the Masur-Veech measure $\mu_2$ from 
$\Hk$ and $(\Om,\P)$ is a space of Markov compacta. The first return map of the flow $P_t$ for an appropriate Poincar\'e section is mapped by $\Xi_\Rk$ to the shift map $\sigma$ on $(\Omega,\P)$. This correspondence maps the Rauzy-Veech cocycle over the Teichm\"uller flow into the renormalization cocycle for the Markov compacta. 
Moreover, the map $\Xi_\Rk$ induces a map defined for a.e.\ $\Xk\in \Vk(\Rk)$, from the corresponding Riemann surface $M(\Xk)$ to a Markov compactum $X(\bom)\in {\Omega}$, intertwining their 
vertical and horizontal flows.


Now let $f$ be a Lipschitz function on $M$ with an abelian differential $\omb$. Under the symbolic coding from \cite{Buf-umn}, it is mapped into a
weakly Lipschitz function on $X(\bom)$ and then to a weakly Lipschitz function $f^{\vec{s}}$ on $\Xxs$ with  $\vec{s}$ given by Lemma~\ref{symb-flow}. By definition, its norm $\|f^{\vec{s}}\|_L$ is dominated by $\|f\|_L$ for all $\vec{s}$. 
Once we check all the assumptions, we can apply Theorem~\ref{th-main1} and obtain the H\"older property of the spectrum of the suspension flow for $\P$-a.e.\ $\bom\in \Om$, for Lebesgue-a.e.\ $\vec{s}$, which in view of the mutual absolute continuity indicated above, is equivalent to the  H\"older property for the flow $h_t^+$, as desired.
Note that the dependence of $r_0$ on $\beta$ (determined by
$\bom$ and $\vec{s}$) in Theorem~\ref{th-main1} will be subsumed by the dependence of $r_0$ on $\bom$ in
Theorem \ref{main-moduli}.

In order to reduce Theorem \ref{main-moduli} to Theorem~\ref{th-main1}, we must now check that the assumptions of Theorem \ref{th-main1} hold for  the left shift on the space of Markov compacta endowed with the push-forward of the Masur-Veech smooth measure under the isomorphism of \cite{Buf-umn}. It is clear that condition (C1) follows from (b), which we discuss below. Condition (C2) holds because the renormalization matrices in the
Rauzy-Veech induction all have determinant $\pm 1$. Condition (C3) holds by a theorem of Zorich \cite{Zorich}.
The condition (a) on the Lyapunov spectrum from Theorem~\ref{th-main1} follows from results of Forni \cite{F2} in 
our case
(later Avila and Viana \cite{AV} proved this for an arbitrary genus $\ge 2$).

In order to ensure conditions (b) and (c), 
we need to recall 
the construction of the Markov compactum and Bratteli-Vershik realization of the translation flow from \cite{Buf-umn}.
The symbolic representation of the translation flow on the 2-sided Markov compactum is obtained as the natural extension of the 1-sided symbolic representation for the IET which we now describe. An interval exchange is denoted by $(\lam,\pi)$, where $\pi$ is the permutation of $m$ subintervals and $\lam$ is the vector of their lengths. The well-known Rauzy induction (operations ``a'' and ``b'') proceeds by inducing on a smaller interval, so that the first return map is again an exchange of $m$ intervals. The Rauzy graph
is a directed labeled graph, whose vertices are permutations of IET's and the edges lead to permutations obtained by applying one of the operations. Moreover, the edges are labeled by the type of the operation (``a'' or ``b''). As is well-known, for almost every IET, there is a corresponding infinite path in the Rauzy graph, and the length of the interval on which we induce tends to zero. For any finite ``block'' of this path, we
have a pair of intervals $J\subset I$ and IET's on them, denoted $T_I$ and $T_J$, such that both are
exchanges of $m$ intervals and $T_J$ is the first return map of $T_I$ to $J$. Let $I_1, \dots, I_m$ be the subintervals of the exchange $T_I$ and $J_1, \dots, J_m$ the subintervals of the exchange $T_J$. 
Let $r_i$ be the return time for the interval $J_i$ into $J$ under $T_I$, that is,
$
r_i=\min\{ k>0: T_I^kJ_i\subset J\}.
$
Represent $I$ as a Rokhlin tower over the subset $J$ and its induced map $T_J$, and  write 
$$I=\bigsqcup\limits_{i=1,\dots, m, k=0, \dots, r_i-1} T^{k}J_i.
$$
By construction, each of the ``floors'' of our tower, that is, each of the subintervals $T_I^{k}J_i$, is a subset of some, of course, unique,  subinterval  of the initial exchange, and we define an integer $n(i,k)$ by the formula
$$
T_I^{k}J_i\subset I_{n(i,k)}.
$$
To the pair $I,J$ we now assign a substitution $\zeta_{IJ}$ on the alphabet 
$\{1, \dots, m\}$ by the formula 
\begin{equation} \label{zeta-def}
\zeta_{IJ}: i\to n(i,0)n(i,1)\dots n(i, r_i-1).
\end{equation}
This is the sequence of substitutions arising from the Bratteli-Vershik realization of an IET. 

\medskip

\noindent {\bf Remark.} Veech \cite{veechmmj} proved that if the IET satisfies the Keane condition and is uniquely ergodic, then it is uniquely determined by the sequence of renormalization matrices arising from the Rauzy-Veech induction. This implies that the map (\ref{map-veech}) is injective on the set of zippered rectangles of full measure (those which correspond to uniquely ergodic horizontal and vertical flows).

\medskip

Condition (c) is verified in the next lemma.  
Words obtained from finite paths in the Rauzy graph will be called admissible. 

\begin{lemma} \label{lem-combi1}
There exists an admissible word $\bq$, which is {\bf simple}, whose associated matrix $A(\q)$ has strictly positive entries, and the corresponding substitution $\zeta$, with $Q = \Sf_\zeta=A(\q)^t$ having the property that there exist  {\bf good return words}
$u_1,\ldots,u_m\in GR(\zeta)$, such that $\{\vec{\ell}(u_j):\ j\le m\}$ is a basis of $\R^m$.
\end{lemma}

\medskip

We start with a preliminary claim.
\begin{lemma}\label{ex-zeta}
There exists a letter $c$ and an admissible simple word $W$ such that $\eta(j)$ starts with $c$, for all letters $j\le m$, where $\eta = \eta(W)$ is the  substitution associated to $W$.
\end{lemma}

\begin{proof} Indeed, start with an arbitrary loop $V$ in the Rauzy graph such that the corresponding renormalization matrix has all entries positive. Consider the  interval exchange transformation with periodic 
Rauzy-Veech expansion obtained by going along the loop repeatedly (it is known from \cite{veech} that such an IET exists).
As the number of passages through the loop grows, the length of the interval forming phase space of the new interval exchange (the result of the induction process) goes to zero. In particular, after sufficiently many moves, this interval will be completely contained in the first subinterval of the initial interval exchange --- but this  means, in view of (\ref{zeta-def}) that $n(i,0)=1$ for all $i$, and hence the resulting substitution $\zeta(V^n)$ has the property that $\zeta(V^n)(j)$ starts with $c=1$ for all $j$. It remains to make sure that the admissible word is simple.
Observe that concatenating two loops $V_1$, $V_2$ starting at the same vertex we obtain $\zeta(V_1V_2)=\zeta(V_2)\zeta(V_1)$. If $\zeta(j)$ starts with $c$ for all $j$, then $\xi\zeta(j)$ starts with the first letter of $\xi(c)$ for all $j$. Thus, we can make sure that $V_1=V^n$ starts and ends with the same Rauzy operation symbol --- either $a$ or $b$, by appending another loop at the end.
We can then take $V_2$ to be the loop of the other Rauzy operation symbol starting at the same vertex. As a result, we obtain  the word in the alphabet $\{a,b\}$ corresponding to 
$W:= V_1V_2$ has the form $a\wt V_1 ab^k$ or $b\wt V_1b a^k$, which is obviously simple. The proof is complete. We are using here the fact that in the Rauzy graph there are both
$a$- and $b$-cycles starting at every vertex.
\end{proof}

\medskip

\begin{proof}[Proof of Lemma~\ref{lem-combi1}]
Let $W$ be the admissible word in the Rauzy graph given by Lemma \ref{ex-zeta}. By construction, the matrix  $\Sf_\eta$ of the substitution $\eta=\eta(W)$ has all entries strictly positive, hence $\eta(i)$ contains all letters $j$, for any $i\le m$.
We can always replace the substitution $\eta$ by its positive power $\eta^k$, since $\eta$ corresponds to a loop in the Rauzy graph. Note that, for every $i\le m$, the word
$\eta^2(i)$ is a concatenation of {\em all} words $\eta(j)$, $j\le m$, in some order, maybe with repetitions,  all of which begin with 
$c$. Therefore, for every $i\le m$, the word $\eta^3(i)$ contains every $\zeta(j)$, $j\le m$, followed by another $\zeta(j')$, also starting with $c$. It follows that $u_j:=\eta(j)$ is a good return word for $\zeta:=\eta^3$, for every $j\le m$. The population vector
$\vec{\ell}(\eta(j))$ is the $j$-th column vector of $\Sf_\eta$. As is well-known, the
matrices corresponding to Rauzy operations are invertible, which implies that the columns of $\Sf_\eta$ span $\R^m$. Finally, note that if $W$ is simple, then $\q=WWW$ is simple as well, 
and it has all the desired properties. The proof is complete. 
\end{proof}

The only remaining, key condition to check is (\ref{eq-star1}).  It will be derived from a variant of the the exponential estimate for return times of the Teichm{\"u}ller flow into compact sets. 
For large compact sets of special form, this estimate is due to Athreya \cite{Ath},  whereas in the general form it was established in \cite{buf-jams} 
and independently in \cite{AGY}.
We will mostly use the same notation as above, but indicate the correspondence with the notation of \cite{buf-jams}. The symbolic coding of the Rauzy-Veech induction map on the space of interval exchange transformations used in \cite{buf-jams}   corresponds to the symbolic coding of the Teichm{\"u}ller flow as a suspension flow over the shift on the space of Markov compacta: indeed, the Rauzy-Veech expansion  precisely identifies an interval exchange transformation with a Bratteli-Vershik automorphism (cf. \cite{Buf-umn}). 

The symbol 
$\Delta(\Rk)$ stands for the space of interval exchange transformations 
whose permutation lies in a given Rauzy class  $\Rk$ (fixed and omitted  from notation); the symbolic space 
$\Omega^+$ is the one-sided topological Markov chain over a countable alphabet that realizes the symbolic coding of the Rauzy-Veech induction map; the space 
$\Omega$ is its natural extension, the  corresponding 
two-sided topological Markov chain.  
The space $\Omega$ can also be viewed as the phase space of the natural extension of the Rauzy-Veech induction, that is, the space of
 sequences of interval exchange transformations ordered by nonpositive integers:
$$
(\la(0), \pi(0)), (\la(-1), \pi(-1)), \dots, (\la(-k), \pi(-k)), \dots)
$$
where $(\la(n), \pi(n))$ is the image of $(\la(n-1), \pi(n-1))$ under 
the Rauzy-Veech induction map, in particular
$$
\la(n)=\frac{\Af_n\la(n+1)}{|(\Af_n\la(n+1)|},
$$
\noindent
where $\Af_n$ is the  
renormalization matrix $A_n$ from Section 4. The symbol $|\la|$ stands for the sum of coordinates of a vector $\la$. In other words,
the induction map that takes $(\la(n), \pi(n))$ to  $(\la(n+1), \pi(n+1))$ consists in applying the matrix $\Af_n^{-1}$ and normalizing to unit length. Following \cite{buf-jams}, we also introduce the ``non-normalized'' 
lengths $\La(n)$ inductively by the rule $$\la(0)=\La(0),\ \  \La(n)=\Af_{n+1}\La(n+1)\ \ \mbox{ for}\  n<0.$$
Informally, $\log|\Lam(-n)|$ is the ``Teichm\"uller time'' corresponding to the discrete normalization ``Rauzy-Veech'' time $n$.

 Let $\bq=q_1\ldots q_k$ be a simple word admissible in the Rauzy class, such that the resulting Rauzy-Veech renormalization matrix  has all entries positive. We further let $\Om_\bq= [\bq.\bq]$. 
The symbol $\Prob$ denotes the push-forward of the Masur-Veech measure.
For 
$\bom\in \Omega_\bq$,
let $\ell_{\q}(\bom)$ be the   
 return time to the cylinder set $[\bq,\bq]$ in the negative direction, i.e.
 $$
 \ell_\bq(\bom) = \min\{n\ge 1:\ \bom_{-n-k+1}\ldots \bom_{-n} = \bom_{-n+1}\ldots\bom_{-n+k} = \bq\}.
 $$
Further, denote by $L_\q(\bom)$ the corresponding ``Teichm\"uller time'':
$$
L_\bq(\bom) = \log|\Lam(-\ell_\bq(\bom))|.
$$

\begin{prop}
\label{expflow}
There exists $\epsilon>0$ such that for any $\bom^+ \in \Om^+_\bq$ we have 
$$
\int_{\Omega} 
\exp(\epsilon L_{\q}(\bom))\,d\Prob(\bom|\bom^+)<+\infty.
$$
\end{prop}

Comparing the definitions, we see that
this is equivalent to (\ref{eq-star1}), so the remaining condition (d) in Theorem~\ref{th-main1}
will follow from Proposition~\ref{expflow}.

\begin{proof}[Proof of  Proposition~\ref{expflow}] 
We argue in much the same way as in  Section 11 of \cite{buf-jams} (compare with the arguments involving the ``bounded distortion'' condition of \cite{AV}; see also
Section 5.1 of \cite{AD}).
Our main tool will be Lemma 16 from \cite{buf-jams} whose formulation we now recall in our notation:

\begin{lemma}
\label{flowtime}
For any admissible word $\q$ such that all entries of the matrix $A(\q)$
are positive, there exist constants $K_0(\q) ,p(\q)$, depending only on
$\q$ and  such that the following is true.
For any $K\geq K_0$ and any $(\la,\pi)\in\Delta(\Rk)$,
$$
\Prob\bigl(\exists n: (\la(-n),\pi(-n))\in [\q.\q],\ 
|\La(-n)|<K)\,|\,(\la,\pi)=(\la(0),\pi(0))\bigr)\geq p(\q).
$$
\end{lemma}

\noindent {\bf Remark.} In fact, the paper \cite{buf-jams} proves this lemma for the coding of the Rauzy-Veech-Zorich induction, that is, for the ``Zorich acceleration'' of the Rauzy-Veech induction, see \cite{buf-jams} for details. However, the proof immediately transfers since it is stated in terms of the ``Teichm\"uller time'' $\log|\Lam(-n)|$, which is invariant under the acceleration. To avoid potential problems with ``overlapping occurrences'' we  choose the word $\q$ to be simple in the Rauzy alphabet $\{a,b\}$, as we do in Lemma~\ref{ex-zeta}.

\medskip

The next lemma is a reformulation of Lemma 7 from \cite{buf-jams}.

\begin{lemma}\label{lem7}
There exists $K_0\in \N$ and $C=C(\Rk)\ge 1$ such that the following holds for any $K>K_0$:
for all $\bom^+\in \Om^+$ we have
$$
\Prob\bigl(|\Lam(-1)|>K\,|\,\bom^+\bigr) \le \frac{C}{K}\,.
$$
\end{lemma}

Define a random time $k_1(\bom)$ to be
the first moment $n\ge 1$ such that $|\La(-n)(\bom)|>K$.
Note that the map
$$
{\wt \sigma}(\bom)=\sigma^{-k_1(\bom)}(\bom)
$$
is invertible.
Introduce a function $\eta: \Omega\to{\mathbb N}$ by
the formula
$$
\eta(\bom)=\left[\frac{\log |\La(-k_1(\bom))|}{\log K}\right].
$$
In other words, $\eta(\bom)=n$ if
$
K^n\leq |\La(-k_1(\bom))|< K^{n+1}.
$
Combining Lemmas \ref{flowtime} and \ref{lem7} we obtain

\begin{prop} \label{opa}
For any $\bom^+\in \Om^+$ and any $r\in \N$ we have
$$
\Prob\bigl(\{\bom:\ \eta(\bom) = r,\ \bom_{-k_1(\bom)+1},\ldots,\bom_0\ \mbox{does not contain $\q$}\,|\,
\bom^+\}\bigr) \le \frac{C}{K^{r-1}}\cdot (1-p(\q)).
$$
\end{prop}



Now, take a large $N$ and let
$$
n_N(\bom)=\min \{n:\ \eta(\bom)+\dots +\eta ({\wt \sigma}^{n}\bom)\geq 
N\}.
$$
Let also
$$
k_n(\bom) = k_1(\bom) + \cdots + k_1({\wt \sig}^n\bom).
$$
Then we have, by the definition of $\eta$:
\begin{equation} \label{Lam}
K^N\le |\Lam(-k_{n_N(\bom)})|< K^{n_N(\bom) + \eta(\bom)+\dots +\eta ({\wt\sigma}^{n}\bom)}.
\end{equation}
Clearly $n_N(\bom) \le N$, and observe that for all $\bom^+\in \Om^+$,
$$
\Prob(\Bk(N)\,|\,\bom^+) \le \frac{C}{K^{N}}\,,\ \ \mbox{where}\ \ \Bk(N) = \{\bom:\ \eta({\wt \sig}^{n_N(\bom)}\bom)> N\},
$$
by Lemma~\ref{lem7}.
Consider the set
$$
{\wt \Omega(N)}=
\{\bom:\ \bom_{-k_{n_N(\bom)}+1},\dots, \bom_{0}\ \ {\rm does \ not \
contain \ the
\ word \ } \q\}\ \cap \{\bom:\ \eta({\wt\sig}^{n_N(\bom)}\bom)\le N\}.
$$
Note that
$$
\{\bom: L_{\q}(\bom)>3N \}\subset {\wt \Omega(N)} \cup \Bk(N).
$$
It suffices, therefore, to prove that there exists $\rho<1$ such that
$$
\Prob({\wt \Omega}(N)\,|\,\bom^+)\leq \rho^N.
$$
We have from Proposition~\ref{opa}, for any $\ell\in \N$ and $n_1,\ldots,n_\ell\in \N$:
\begin{eqnarray}
& & \Prob\bigl( \eta(\bom) = n_1,\ \eta(\wt{\sig}\bom) = n_2,\ldots, \eta(\wt{\sig}^\ell\bom) = n_\ell,\
\bom_{-k_\ell(\bom)+1},\ldots,\bom_0\ \mbox{does not contain}\ \q\,|\,\bom^+\bigr)\nonumber \\
& & \le  (1-p(\q))^\ell \left( \frac{C}{K}\right)^{(n_1 + \cdots + n_\ell)-\ell}. \label{tutu}
\end{eqnarray}
Choose $K\in \N$ such that $1-p(\q) + \frac{C}{K}<1$. We have
\begin{eqnarray*}
\Prob(\wt{\Om}(N)\,|\,\bom^+) & \le & \sum_{\wtil{N}=N}^{2N} \Prob(\wt{\Om}(N)\,|\,\bom^+ \ \&\ 
\eta(\bom) + \cdots + \eta(\wt{\sig}^{n_N(\om)}\bom) = \wtil{N}) \\
& \le & \sum_{\wtil{N}=N}^{2N} \left(1-p(\q) + \frac{C}{K}\right)^{\wtil{N}} \le \rho^N,
\end{eqnarray*}
using (\ref{tutu}) and the binomial formula in the last line,
and Proposition \ref{expflow} is proved.
This concludes the proof of Theorem \ref{main-moduli} as well.
\end{proof}


\end{document}